\documentclass[reqno,11pt]{amsart}
\usepackage{amsthm,amsfonts, amssymb, amscd}
\usepackage{euscript}

\newtheorem{Thm}[equation]{Theorem}
\newtheorem{Lem}[equation]{Lemma}
\newtheorem{Cor}[equation]{Corollary}
\newtheorem{Prop}[equation]{Proposition}

\theoremstyle{remark}
\newtheorem{Rem}[equation]{Remark}
\theoremstyle{remark}

\theoremstyle{definition}
\newtheorem{Def}[equation]{Definition}

\numberwithin{equation}{section}

\newcommand{\C}{\mathbb{C}}           
\newcommand{\Z}{\mathbb{ Z}}           
\newcommand{\Ad}{\operatorname{Ad }}             
\newcommand{\ad}{\operatorname{ad}}             

\newcommand{\End}{\operatorname{End}}
\newcommand{\Hom}{\operatorname{Hom}}

\newcommand{\rank}{\operatorname{rank}}

\renewcommand{\ker}{\operatorname{ker }}

\newcommand{\im}{\operatorname{im }}
\newcommand{\Gr}{\operatorname{Gr }}

\newcommand{\codim}{\operatorname{codim}}
\newcommand{\pa}{\partial}
\newcommand{\pastar}{{\pa}^{adj}}

\newcommand{\Cstar}{\C^*} 
\newcommand{\fb}{{\mathfrak b}}

\newcommand{\fg}{{\mathfrak g}}
\newcommand{\fh}{{\mathfrak h}}
\newcommand{\fk}{{\mathfrak k}}
\newcommand{\fl}{{\mathfrak l}}

\newcommand{\fp}{{\mathfrak p}}

\newcommand{\fr}{{\mathfrak r}}

\newcommand{\fu}{{\mathfrak u}}

\newcommand{\fz}{\mathfrak z}

\newcommand{\ga}{\alpha}
\newcommand{\gb}{\beta}
\newcommand{\gd}{\delta}
\newcommand{\gD}{\Delta}
\newcommand{\gDI}{{\Delta}_I}
\newcommand{\gre}{\epsilon}
\renewcommand{\gg}{\gamma}
\newcommand{\gG}{\Gamma}

\newcommand{\gl}{\lambda}
\newcommand{\gL}{\Lambda}

 \newcommand{\ch}{\mathcal{H}}
 
 \newcommand{\cl}{\mathcal{L}}
 \newcommand{\cm}{\mathcal{M}}
 \newcommand{\co}{\mathcal{O}}
 \newcommand{\cs}{\mathcal{S}}

 \newcommand{\cg}{\mathcal{G}}
 
 \newcommand{\cb}{\mathcal{B}}
  \newcommand{\cz}{\mathcal{Z}}

\newcommand{\flI}{{\mathfrak l}_I}
\newcommand{\fzI}{{\mathfrak z}_I}
\newcommand{\fpI}{{\mathfrak p}_I}
\newcommand{\fuI}{{\mathfrak u}_I}
\newcommand{\fuIminus}{{\mathfrak u}_{I,-}}
\newcommand{\fpIminus}{{\mathfrak p}_{I,-}}
\newcommand{\frI}{{\mathfrak r}_I}
\newcommand{\flgDm}{\tilde{{\fl}_{\gD}}}
\newcommand{\chstar}{{\ch}^*}

\renewcommand{\tilde}{\widetilde}

\renewcommand{\bar}[1]{\overline{#1}}

\begin{document}
\parskip=4pt
\baselineskip=14pt

\title[The Belkale-Kumar cup product]
{The Belkale-Kumar cup product and relative Lie algebra
cohomology}

\author{Sam Evens}
\address{
Department of Mathematics,
University of Notre Dame,
Notre Dame, IN 46556
}
\email{sevens@nd.edu}

\author{William Graham}
\address{
Department of Mathematics,
University of Georgia,
Boyd Graduate Studies Research Center,
Athens, GA 30602
}
\email{wag@math.uga.edu}
\thanks{Evens and Graham were supported by the National Security Agency}

\address{Department of Math, University of British Columbia, 1984 Mathematics
Rd, Vancouver, BC V6T 1Z2}
\email{erichmond@math.ubc.ca}

\maketitle

\begin{center}
\textsc{with an appendix by Sam Evens, William Graham, and Edward Richmond}
\end{center}

\begin{abstract}
We study the Belkale-Kumar family of cup products on the cohomology of
a generalized flag variety. We give an alternative construction of
the family using relative Lie algebra cohomology, and in particular,
identify the Belkale-Kumar cup product with a relative Lie algebra
cohomology ring for every value of the parameter. As a 
consequence, we extend a fundamental disjointness result of Kostant
to a family of Lie algebras. In an appendix, written jointly with
Edward Richmond, we extend a Levi movability result of Belkale and
Kumar to arbitrary parameters.
\end{abstract}


\noindent 

\section{Introduction} \label{s.intro}

In a remarkable 2006 paper \cite{BeKu:06}, Belkale and Kumar 
introduced a new family of cup products on the cohomology
$H^*(X, \C)$, where $X=G/P$ is a generalized flag variety.
Here $P$ is a parabolic subgroup of a complex connected semisimple
group $G$, and we let $m=\dim(H^2(X, \C))$. For each
$\tau \in \C^m$, Belkale and Kumar construct a product
$\odot_\tau$ on the space $H^*(X, \C)$. They show that for generic
$\tau$, the ring $(H^*(X, \C), \odot_\tau)$ is isomorphic to the usual
cup product on $H^*(X, \C)$, while for $\tau=0:= (0, \dots, 0)
\in \C^m$, the degenerate cup product 
 $(H^*(X, \C), \odot_0)$ is closely related to the
cohomology ring of the nilradical of the Lie algebra of $P$.
Their work was motivated by applications to the Horn problem,
and using their deformed cup product when $\tau=0$, they gave
a more efficient solution to the Horn problem for $G$ not of type
$A$.  

The purpose of this paper is to show that
the Belkale-Kumar family of cup products has an intrinsic
definition as the
product on the space of global sections
of a vector bundle defined using relative Lie algebra cohomology.
 As a consequence, we are able to identify
$(H^*(X, \C), \odot_\tau) \cong H^*(\fg_t, \fl_{\gD})$, where
$t\in \C^m$, $\tau = t^2$, $\fg_t$ is a certain Lie subalgebra of
$\fg \times \fg$, and $\fl_{\gD}$ is the diagonal embedding of
a Levi factor $\fl$ of the Lie algebra of $P$ in $\fg \times \fg$.
For generic $\tau$ and for $\tau=0$, this result is proved
in \cite{BeKu:06}, but other cases appear to be new. In addition,
we generalize a well-known disjointness theorem of Kostant
\cite{Kos:63}
(see Theorem \ref{t.disjointness}).

In more detail, we identify $(\Cstar)^m \cong Z$, the center of
a Levi factor $L$ of $P$.  Using an idea 
important in the study of the DeConcini-Procesi compactification,
 for each $t\in \C^m$, 
 we define a Lie subalgebra
$\fg_t$ containing the diagonal embedding $\fl_{\gD}$ of $\fl$
in $\fg \times \fg$, such that if $t \in Z$, then
$\fg_t = \Ad(t,t^{-1})\fg_{\gD}$, where $\fg_{\gD}$ is the
diagonal in $\fg \times \fg$.  We prove that the relative 
Lie algebra cohomology rings $H^*(\fg_t, \fl_{\gD})$ are
the fibers of an algebraic vector bundle on $\C^m$.  The space
$\chstar(\C^m)$ of global sections of this vector bundle
is an algebra over the ring $\C[t_1, \ldots, t_m]$ of polynomial
functions on $\C^m$.  

The Belkale-Kumar family of cup products is in fact defined as a product
$\odot$ on the ring $H^*(X, \C) \otimes_{\C} \C[\tau_1, \ldots, \tau_m]$.  
( Belkale and Kumar work with integer coefficients, but since all
the rings are torsion-free, nothing is lost by using complex coefficients.)
If we set $\tau_i = t_i^2$, we obtain a product, again denoted $\odot$, on the
ring $H^*(X, \C) \otimes_{\C} \C[t_1, \ldots, t_m]$.
The following is the main result of our paper,
and identifies the Belkale-Kumar cup product with a relative Lie
algebra cohomology ring.

\begin{Thm} \label{t.mainintro} (see Theorem \ref{t.BKcomp}) 
The rings $(H^*(X, \C) \otimes_{\C} \C[t_1, \ldots, t_m], \odot)$
and $\chstar(\C^m)$ are isomorphic.
  Hence, if $t \in \C^m$ and
 $\tau = t^2 \in \C^m$, then $H^*(\fg_t, \fl_{\gD})
\cong (H^*(G/P), \odot_{\tau})$. 
\end{Thm}

This theorem may be regarded as an alternative construction of
the Belkale-Kumar cup product. 
We recall some details from \cite{BeKu:06} to explain the difference
between our approach and the definition of the family of cup products
in \cite{BeKu:06}.  Let $B \subset P$ be a Borel subgroup of $G$,
 let $H$ be a maximal torus of $B$ containing $Z$, and
let the positive roots $R^+$ be the set of roots of $H$ in  the Lie algebra of $B$.
We let $\{ \alpha_1, \dots, \alpha_n \}$ be the corresponding simple roots.
Let $R^+(\fl)$ be the set of roots of $R^+$ such that the corresponding
root space is in the Lie algebra of $L$.  Let $W=N_G(H)/H$ be the Weyl
group and let $W^P$ be the set of $w \in W$ such that $w(R^+(\fl)) \subset R^+$.
Number the simple roots so that $\ga_1, \dots,
\ga_m \not\in R^+(\fl)$ and $\ga_{m+1}, \dots, \ga_n \in R^+(\fl)$.
For a root $\ga \in R^+$, we write $\ga = \sum n_i \ga_i$ for
nonnegative integers $n_i$. Let $e^{\ga}$ be the character of $H$
determined by the root $\ga $.  If $t=(t_1, \dots, t_m) \in
(\Cstar)^m \cong Z$, then $e^{\ga}(t)=\prod_{i=1}^m t_i^{n_i}$.
Thus, defining $F_w(t) = \prod_{\ga \in R^+ \cap w^{-1} R^-} e^{\ga}(t)$
for $t\in Z$,
we see that $F_w(t)$ extends to a polynomial function on $\C^m$.
In \cite{BeKu:06}, the Belkale-Kumar  cup product is defined by the formula
\begin{equation}\label{e.introbk}
\gre_u \odot \gre_v = \sum \frac{F_w(t)}{F_u(t) F_v(t)} c_{uv}^w \gre_w.
\end{equation}
To show that the product $\odot_t$ is defined for all $t\in \C^m$,
Belkale and Kumar use methods from geometric invariant theory to prove
that the quotients $\frac{F_w(t)}{F_u(t)F_v(t)}$ are regular on
$\C^m$ when $c_{uv}^w \not= 0$. 

In our approach, we begin with the $\C[t_1, \dots, t_m]$-algebra
$\chstar(\C^m)$.
For each $w\in W^P$, we find a section $[G_w(t)]$ in $\chstar(\C^m)$,
such that these sections form a basis of $\chstar$ over $\C[t_1, \dots, t_m]$,
and such that the product is given by
\begin{equation} \label{e.introrellie}
[G_u(t)] \cdot [G_v(t)] = \sum \frac{F_w(t)}{F_u(t) F_v(t)} c_{uv}^w [G_w(t)].
\end{equation}
Thus, the algebra $\chstar(\C^m)$
has the same structure constants as those given by the Belkale-Kumar product.
It follows from this that $H^*(X, \C) \otimes_{\C} \C[t_1, \ldots, t_m]$,
equipped with the Belkale-Kumar product, coincides with
the cohomology algebra $\chstar(\C^m)$.  In particular, this proves that
$\frac{F_w(t)}{F_u(t) F_v(t)}$ is regular when the Schubert coefficient
$c^w_{uv}$ is nonzero, without recourse to methods from geometric
invariant theory.


Our approach relates the Belkale-Kumar cup
product to geometry of the DeConcini-Procesi compactification. 
 Moreover, while Belkale
and Kumar require a technical argument using filtrations to relate
$H^*(G/P, \odot_0)$ to relative Lie algebra cohomology, in our approach,
this relation follows
directly from our construction.  
However, our construction does not directly relate
to the notion of $L$-movability of Schubert cycles which motivates the
deformed cup product in \cite{BeKu:06}, although we do extend this
relation in the appendix to this paper, which is jointly written
with Edward Richmond. We remark that our work
is partly 
inspired by ideas from Poisson geometry made explicit in \cite{EvLu:99}
and \cite{EvLu:06}.  In future work, we hope to use Poisson geometry to
study the Horn problem and the Belkale-Kumar cup product. Finally,
the relation to the DeConcini-Procesi compactification suggests
that we can extend the Belkale-Kumar cup product to infinity by
 considering the
closure of $\C^m$ in the compactification, which is a toric
variety, at least when $P$ is a Borel subgroup. 

We summarize the contents of this paper. In Section \ref{s.familylie},
we define the Lie subalgebras $\fg_t$ for $t\in \C^m$, and prove some
properties of an action of $Z$ on $\C^m$. In Section \ref{s.rellie},
we introduce the induced relative Lie algebra cohomology bundle and
prove it has constant rank. In Section \ref{s.globalsch}, we introduce
cocycles $\cs_w(t)$ for $t\in \C^m$, and show that the cocycles
$G_w(t)=\frac{\cs_w(t)}{F_w(t)}$ give representatives for a basis of
$H^*(\fg_t, \fl_{\gD})$. We use this last result to prove our main
theorem.  In Section \ref{s.disjoint},
we generalize a disjointedness theorem of Kostant to the Lie
subalgebras $\fg_t$, for $t\in \C^m$. In the Appendix, jointly
written with Edward Richmond, we extend a Levi movability result
of Belkale and Kumar to arbitrary parameter values.

We would like to thank Shrawan Kumar for several inspiring conversations.
The idea for this work arose from an illuminating lecture given by 
Kumar at the 2007 Davidson AMS meeting. The first author would like
to thank the University of Georgia for hospitality during part of the preparation
of this paper. We would also like to thank Bill Dwyer  for a
 useful conversation.

\section{A family of Lie algebras}\label{s.familylie}
In this section, we introduce notation and the family of Lie algebras $\fg_t$
mentioned in the introduction.

\subsection{Preliminaries} \label{s.prelim} 
Let $\fg$ be a semisimple complex Lie algebra,
and $G$ the corresponding  adjoint group,
with $B \supset H$ a Borel and maximal torus, respectively. Let
$\fb$ and $\fh$ denote the Lie algebras of $B$ and $H$. Let
$R$ be the set of roots of $\fh$ in $\fg$, and we choose the positive system
$R^+$ of roots so  that
the roots of $\fb $ are positive.  
Given a root $\ga \in \fh^*$, let $\fg_\ga$ be the corresponding
root space, let $e_\ga \in \fg_\ga$ denote a nonzero
root vector,  and
let $e^{\ga}$ denote the corresponding character of $H$.
Let $\gD = \{ \ga_1, \ldots, \ga_n \}$ be the
simple roots, which give a basis of $\fh^*$.
For a subset $I \subset \{ 1, \dots, n \}$, let
$\gDI = \{ \ga_i : i\in I \}$, let $\flI$ be the
Levi subalgebra generated by $\fh$ and the root spaces
$\fg_{\pm \ga_i }$ for $i\in I$, and let $\fzI$ denote
the center of $\flI$.
Let $\fpI = \fb + \flI$ be the corresponding 
standard parabolic subalgebra
and let $\fu=\fuI$ be the nilradical of $\fpI$.
 Let $R^+(\flI)=\{ \ga \in R^+ : \fg_\ga \subset \flI \}$
and let $R^+(\fuI) = R^+ - R^+(\fl)$ be the roots of $\fuI$.
Let $\fuIminus = \sum_{\ga \in R^+(\fuI) } \fg_{-\ga}$ and
let $\fpIminus = \flI + \fuIminus$ be the opposite nilradical.
For simplicity, we write $\fl=\flI$, $\fp=\fpI$, $\fu=\fuI$,
$\fu_-=\fuIminus$, and so forth when $I$ is fixed.  We let $P, L$ and
$Z$ denote the closed connected subgroups of $G$ with Lie algebras
$\fp$, $\fl$, and $\fz$ respectively.

Given a subspace $V$ of $\fg$, let $V_{\gD}$ denote the image
of $V$ under the diagonal map $\fg \to \fg \times \fg$,
and let $V_{-\gD}$ denote the image of $V$ under the
antidiagonal map $\fg \to \fg \times \fg$, $X \mapsto (X,-X)$.
 We let
$V'$ and $V''$ be the subspaces of $\fg \times \fg$ defined as
$V' = V \times \{ 0 \}$, $V'' = \{ 0 \} \times V$.
If $X \in \fg$, define $X_{\gD}=(X,0)$, $X_{-\gD}=(X,-X)$, and define
  $X'$, $X''$ in $\fg \times \fg$ by
$X' = (X,0)$, $X'' = (0,X)$.  We
 denote the Killing form on $\fg$ or on $\fg \times \fg$ by $(\cdot, \cdot )$..

Let 
\begin{equation} \label{e.def}
\fr=\frI = \fu' \oplus \fu_-''
\end{equation}
and
\begin{equation} \label{e.defrop}
\fr^{op}=\frI^{op} = \fu_-' \oplus \fu''.
\end{equation}
In particular, $[\fu', \fu_-'']=0$. Note that $\fr$ and $\fr^{op}$ are 
the images of two different embeddings of the direct sum Lie algebra
$\fu \oplus \fu_-$ in
$\fg \times \fg$.

\subsection{Schubert classes} \label{s.schubert}
Because we use Lie algebra cohomology to study
the cohomology of the flag variety, we consider homology and
cohomology with complex coefficients.

Let $X = G/P$. Let $W$ be the Weyl group $N_G(H)/H$,
and consider its subgroup $W_P=N_L(H)/H$. Let $W^P = \{ w\in W :
w(R^+(\fl)) \subset R^+ \}$, and recall that the
map $W^P \to W/W_P$ given by $w\mapsto wW_P$ is bijective.
Let $w_0$ denote the long element of the Weyl group and
let $w_{0,P}$ denote the long element of the Weyl group
of $W_P$.  

\begin{Rem}\label{r.klm}
For $w\in W^P$, then $w_0 w w_{0, P}$ is in  $W^P$ 
 by \cite[Theorem 2.6]{KLM:03}
\end{Rem}

If $w \in W$, let $X_w$ denote the closure of the Schubert cell
$B w P$ in $X$.
For an irreducible subvariety $Z$ of dimension $r$ in $X$, 
let $[Z] \in H_{2 r}(X)$ denote the homology class of $Z$.
The set $\{ [X_w]  \ | \ w\in W^P, \ 2 l(w) = k \}$ is a basis for the group $H_k(X)$.
Let $Y_w$ denote the closure of $B_{-} w P$ in $X$.

As $B_- = w_0 B w_0$ and $w_{0,P}$ is represented by an element
of $P$, we see that
$$
B_-  w P = w_0 B w_0 w w_{0,P} P.
$$
Thus, $Y_w = w_0 X_{w_0 w w_{0.P}}$.
As left multiplication by $w_0$ does not change the homology class,
we see that $[Y_w] = [X_{w_0 w w_{0,P}}]$.

As in Belkale-Kumar, let $\{ \gre_w \ | \ w\in W^P, \ 2 l(w) = k \}$
 denote the dual
basis in cohomology (under the Kronecker pairing between
$H^k(X)$ and $H_k(X)$, which is not related to Poincar\'e duality),
i.e.,
$$
\langle \gre_u, [X_v] \rangle = \gd_{u,v}.
$$

There is a Poincare duality map
$$
\cap [X]: H^i (X) \to H_{2N - i}(X),
$$
where $N$ is the complex dimension of $X$.   We have
$$
\gre_w \cap [X] = [Y_w] = [X_{w_0 w w_{0,P}}].
$$
Alternatively,
$$
\gre_{w_0 w w_{0,P}} \cap [X] = [X_w].
$$

Let $\gL_w = \gre_{w_0 w w_{0, P}} \in H^{2N - 2 l(w)}(X)$.  Belkale-Kumar
denote $\gL_w$ by $[\bar{\gL}_w^P]$.

\subsection{A family of Lie algebras} \label{s.subfamily}
If $s \in H$, $X \in \fg$, we often write $s X$ or $s \cdot X$ for $(\Ad s) X$.
 Let $d = \dim(\fg)$, and 
let $\Gr(d, \fg \times \fg)$ denote the Grassmannian
of $d$-dimensional subspaces of $\fg \times \fg$.

For $\ga \in R^+$, let $\ga = \sum_{i=1}^n k_i(\ga) \ga_i$.
Consider the monomial $t_\ga :=  t_1^{k_1(\ga)} \cdots t_n^{k_n(\ga)}
\in \C[t_1, \dots, t_n]$, and note that $t_{\ga_i}=t_i$ for each simple root $\ga_i$.
For $\ga \in R^+$, let $E_\ga(t)=(t_{\ga}^2 e_\ga, e_\ga)$ and 
$E_{-\ga}(t)=(e_{-\ga}, t_{\ga}^2 e_{-\ga})$.

 For $t$ in $\C^n$, let 
\begin{equation} \label{e.gtdefn}
\fg_t := \fh_{\gD} + \sum_{\ga \in R^+} (\C E_\ga(t)
+ \C E_{-\ga}(t)).
\end{equation}
It is routine to verify that $\fg_t \in \Gr(d, \fg \times \fg)$.

\begin{Lem} \label{l.cnembedding}
The map $F:\C^n \to \Gr(d, \fg \times \fg)$ given by $F(t)=\fg_t$
 is a morphism of
 algebraic varieties. 
\end{Lem}

\begin{proof}
Let $t_0 := (0,\dots, 0)$ denote the origin of $\C^n$, and
note that 
$$
\fg_{t_0} = \fh_{\gD} + \sum_{\ga \in R_+} (\C (0, e_\ga)
+ \C (e_{-\ga}, 0)),$$
and let
$$
\fg_- := \fh_{-\gD} + \sum_{\ga \in R_+} (\C (e_\ga, 0)
+ \C ( 0, e_{-\ga})).
$$
Consider the affine open set
$$
{\Gr}_{\fg_-} = \{ U \in \Gr(d,\fg \times \fg) : U \cap \fg_- = 0 \}
$$
of $\Gr(d,\fg \times \fg)$.
Then $F(\C^n) \subset {\Gr}_{\fg_-}$ and
the map $\Gamma:\Hom(\fg_{t_0}, \fg_-) \rightarrow {\Gr}_{\fg_-}$ given by
$\phi \mapsto \{ X+ \phi(X) : X \in \fg_{t_0} \}$ is an isomorphism
of affine varieties. To prove the lemma, it suffices to verify that
the map $\tilde{F}:\C^n \to \Hom(\fg_{t_0}, \fg_-)$ such that
$\Gamma \circ \tilde{F} = F$ is a morphism.
For this, we let $h(1), \dots, h(n)$ be a basis of $\fh$,
let $R^+ = \{ \gb_1, \ldots, \gb_r \}$,
 give $\fg_{t_0}$ the basis 
$$
 \{ e_{\gb_1}^{\prime\prime}, e_{-\gb_1}^{\prime},   \ldots,   e_{\gb_r}^{\prime\prime}, 
e_{-\gb_r}^{\prime} \}  \cup
\{ {h(1)}_{\gD}, \dots, {h(n)}_{\gD} \},
$$
and give $\fg_-$ the basis
$$
 \{ e_{\gb_1}^{\prime}, e_{-\gb_1}^{\prime\prime},   \ldots,   e_{\gb_r}^{\prime}, 
e_{-\gb_r}^{\prime\prime} \}  \cup
\{ {h(1)}_{-\gD}, \dots, {h(n)}_{-\gD} \}
$$
With respect to these bases $\tilde{F}(t)$ is a diagonal matrix
with entries
$$
(t_{\gb_1}^2, t_{\gb_1}^2, \dots, t_{\gb_r}^2, t_{\gb_r}^2, 0, \dots, 0).
$$
where $R_+=\{ \ga_1, \dots, \ga_r \}$, and the Lemma follows.
\end{proof}

\begin{Rem}
The map $F:\C^n \to \Gr(d, \fg \times \fg)$ is a composition of
two morphisms from \cite{DePr:83}considered by DeConcini and Procesi
in their study of the wondeful
compactification, from which one can derive the Lemma 
(see \cite[Lemma 2.7 and Proposition 3.7]{EvJo:08}).
\end{Rem}

We identify the maximal torus $H$ as a subset of $\C^n$ via the morphism
$$
H \stackrel{\cong}{\rightarrow} (\Cstar)^n \hookrightarrow \C^n
$$
\begin{equation} \label{e.HintoCn}
s \mapsto (e^{\ga_1}(s), \ldots e^{\ga_n}(s)) = (s_1, \dots, s_n).
\end{equation}
We denote $1 := (1, \dots, 1) \in \C^n$, and if 
 $a=(a_1, \dots, a_n)$ and $b=(b_1, \dots, b_n)$,
we let $ab = (a_1 b_1, \dots, a_n b_n) \in \C^n$.

Note that $H$ acts on $F(\C^n)$ via the action
\begin{equation} \label{e.fgt}
s\cdot \fg_t = (s, s^{-1}) \fg_t = \{ ( (\Ad s)X, (\Ad s^{-1})Y) \ | (X, Y) \in \fg_t \}.
\end{equation}

\begin{Rem} \label{r.gtisom}
For $s\in H$ and $t\in \C^n$,
$$
s\cdot \fg_t = \fg_{s t}.
$$
Indeed, this follows easily from equations \eqref{e.gtdefn} and \eqref{e.fgt}.
\end{Rem}

If $\{ x_1, \dots, x_n \}$ is a basis of $\fh$,
$\fg_t$ has basis given by 
  \begin{equation} \label{e.basisgt}
  \{ (x_i, x_i) \}_{i =1, \ldots, n} \cup \{  E_{\ga} (t), \  E_{-\ga}(t) \}_{\ga \in R^+}.
  \end{equation}

%

\begin{Rem}\label{r.indexconvention}
Let $\dim(Z)=m$ and order the simple roots $\ga_1, \dots, \ga_n$
so that the roots $\ga_1, \dots, \ga_m \in R^+(\fu)$ and $\ga_{m+1}, \dots, \ga_n \in R^+(\fl)$.
 Then using equation \eqref{e.HintoCn}, we identify $Z \cong (\Cstar)^m$
and $\overline{Z} = \C^m := \{ (z_1, \dots, z_m, 1, \dots, 1) \in \C^n \}$.
We denote by $0$ the point $(0, \dots, 0) \in \C^m$, which is the same as
$$(\underbrace{0, \dots, 0}_{\mbox{m}}, \underbrace{1, 1, \dots, 1}_{\mbox{n-m}}) \in \C^n.$$
\end{Rem}

\begin{Rem}\label{r:fgspecial}
It follows from definitions that $\fg_1 = \fg_{\gD}$ and
$\fg_0 = \fl_{\gD} + \fr^{op}$. By Remark \ref{r.gtisom}, we deduce
that $\fg_s = s\cdot \fg{\gD}$ for $s\in H$.
\end{Rem}

\begin{Lem} \label{l.gylielevi}
For each $t\in \C^m$, $\fg_t$ is a Lie subalgebra of $\fg \times \fg$
and $\fl_{\gD} \subset \fg_t$.
\end{Lem}

\begin{proof}
By Remark \ref{r:fgspecial}, , $\fg_{s} = \Ad(s,s^{-1})\fg_1$ for $s\in H$.
Hence, $\fg_{s}$ is the image of the Lie subalgebra $\fg_1 = \fg_{\gD}$ under a Lie algebra
automorphism and hence is a Lie
subalgebra. It follows that $\fg_t$ is a Lie subalgebra for all
$t\in \C^n$ since the set of Lie subalgebras of $\Gr(d,\fg \times \fg)$
is a closed subvariety. For $z\in Z$, $\fl_{\gD} = \Ad(z,z^{-1})\fl_{\gD}
\subset \Ad(z,z^{-1})\fg_{\gD}$. The second claim now follows since
$Z$ is dense in $\C^m$ and
the variety of planes in  $\Gr(d, \fg \times \fg)$ containing $\fl_{\gD}$ is a closed subvariety.
\end{proof}

The elements $\{ E_{\ga}(t), E_{-\ga}(t) \}$ for
$\ga \in R^+(\fu)$ give a basis of $\fg_t/\fl_{\gD}$ via
projection. Let $\{ \phi_{\ga}(t), \phi_{-\ga}(t) \}$ for
$\ga \in R^+(\fu)$ be the dual basis.


\begin{Def}
Let $\cg$ be the universal vector bundle on $\Gr(d, \fg \times \fg)$,
so that for a point $x\in \Gr(d, \fg \times \fg)$ corresponding to
a $d$-dimensional subspace $U$, the fiber ${\cg}_x = U$. 
 Let $\flgDm$ denote the trivial
vector bundle $F(\C^m) \times \fl_{\gD}$.
By Lemma \ref{l.gylielevi}, $\flgDm$ is a subbundle of the pullback
bundle $F^*\cg$,
and we let $E = (F^*\cg/\flgDm)^*$ denote the dual of the quotient
bundle over $\C^m$.
\end{Def}

For $t\in \C^m$, consider the morphism $f_t:\fr \to (\fg_t/\fl_{\gD})^*$
defined by 
$$
 \langle f_t(X), M + \fl_{\gD} \rangle = (X, M),
 $$
for $X \in \fr$, $M \in \fg_t$,

For $s\in Z$, let 
\begin{equation} \label{e.gammas}
\Gamma_s:\fr \to \fr
\end{equation}
 denote the restriction 
to $\fr$ of the automorphism $(s, s^{-1})$
 of $\fg \times \fg$.  Then $\gG_s$ satisfies the
following equations (for $\ga>0$):
$$
\gG_s(e_{\ga}') = e^{\ga}(s) e_{\ga}'
$$
and
$$
\gG_s(e_{-\ga}'') = e^{\ga}(s) e_{-\ga}''.
$$

Since $(s,s^{-1}):\fg_t \to \fg_{s t}$ for $t\in \C^m$,
the dual $(s,s^{-1})^*$ maps $\fg_{s t}^* \to \fg_t^*$.

\begin{Lem} \label{l.fgamma}  
$$
(s,s^{-1})^* \circ f_{s t} = f_t \circ \gG_{s^{-1}}.
$$
\end{Lem}

\begin{proof}
By definition, if $x \in \fr$, $u \in \fg_t$, then
\begin{align*}
\langle (s, s^{-1})^{*} \circ f_{s t} (x), u \rangle   = 
& \langle f_{s t}(x), (s, s^{-1})u \rangle 
  = (x, (s, s^{-1})u)  \\
  =  & ((s^{-1}, s)x, u) = \langle f_t \circ \gG_{s^{-1}}( x), u \rangle.
\end{align*}
\end{proof}

For $X\in \fr$, let $f_X$ denote
the section of the vector bundle $E$ over $\C^m$ such that
$f_X(t)=f_t(X)$. By the following result, $f_X$ is nowhere
vanishing if $X\not= 0$.

\begin{Prop} \label{p.iso2}
 For each $t \in \C^m$, the map $f_t$ is an isomorphism. In particular,
$E$ is isomorphic to the trivial bundle $\C^m \times \fr$ over $\C^m$.
 \end{Prop}

\begin{proof}
For $\ga \in R^+(\fu)$, let $c_{\ga} = (e_{\ga}, e_{-\ga})$.
It is straightforward to check that for $t\in \C^m$,
$f_t(e_{\ga}') = c_{\ga} \phi_{-\ga}(t)$ and
$f_t(e_{-\ga}'') = c_{\ga} \phi_{\ga}(t)$.  Since
the elements $\{ e_{\ga}', e_{-\ga}'' \}, \ga \in R^+(\fu)$
and $\phi_{\pm\ga}(t), \ga \in R^+(\fu)$ are bases of
$\fr$ and $(\fg_t/ \fl_{\gD})^*$ respectively, the first assertion follows. The
second assertion follows since the morphism $(t, X)\mapsto (t,f_t(X))$
is an isomorphism of vector bundles.
\end{proof}

For a subset $J$ of $\{ 1, \ldots, m \}$,
 let $p_J \in \C^m$ 
denote the point whose coordinates $t_j$ are equal to $1$ for
 $j \in J$, and $0$ if $j\not\in J$.  In particular, $1 = p_{ \{ 1,\dots, m \} }$
and $0 = p_{  \emptyset  }$. 
 The $p_J$ form a set of representatives of $Z$-orbits in $\C^m$.  
 We now describe $\fg_{p_J}$ for arbitrary $J$.

 \begin{Prop} \label{p.iso3}
 For $J \subset \{ 1, \ldots, m \}$,  let $\fl_{I\cup J}$ denote the Levi subalgebra
 of $\fg$ spanned by $\fh$ and $e_{\ga}$, for $\ga$ in the root system
 generated by $\ga_j \ (j \in I \cup J)$.  Let $\fu_{I \cup J,+}$ 
 (resp.~$\fu_{I\cup J,-}$) denote the
 span of the positive (resp.~negative) root spaces not in $\fl_{I\cup J}$.
 Then
 $$
 \fg_{p_J} = \fl_{I\cup J,\gD} \oplus \fu_{I\cup J,-}' \oplus \fu_{I\cup J,+}''.
 $$
 \end{Prop}
 
 \begin{proof}
 This is a straightforward calculation.
 \end{proof}

\section{The relative Lie algebra cohomology bundle} \label{s.rellie}

In this section, for $t\in \C^m$, we study the family of relative Lie algebra
cohomology spaces $H^*(\fg_t, \fl_{\gD})$ and show that the
dimension is independent of $t$.

The diagonal action of $L$ on $\fg \times \fg$ preserves
$\fr$ and each $\fg_t$, and the map $f_t:  \fr \to (\fg_t/\fl_{\gD})^*$
is $L$-equivariant.  We consider the cochain complex
$$
(C^{\cdot} (\fg_t, \fl_{\gD}) , d_{\fg_t})
$$
where $C^i(\fg_t, \fl_{\gD}) =  (\bigwedge^i ((\fg_t/\fl_{\gD})^*))^L$,
and the differential is the relative Lie algebra cohomology
differential corresponding to the
trivial representation of $\fg_t$ \cite[I.1.1]{BoWa:00}.  We omit the trivial
representation from the notation and denote the cohomology
of this complex
by $H^*(\fg_t, \fl_{\gD})$.   Observe that $\fg_1 = \fg_{\gD}$,
 so $H^*(\fg_1, \fl_{\gD})$ is canonically identified with
$H^*(\fg, \fl)$.  Also, $\fg_0 = \fr^{op} \oplus \fl$ by Remark \ref{r:fgspecial}. 
 This is not a direct sum
of Lie algebras, but $\fr^{op}$ is an ideal in $\fg_0$, 
and $H^*(\fg_0, \fl_{\gD})$ is identified
with $H^*(\fr^{op})^L \simeq H^*(\fu \oplus \fu_-)^L$.

Define $C^i(\fr) =  (\bigwedge^i \fr)^L$.
The $L$-equivariant isomorphism $f_t : \fr \to (\fg_t/\fl_{\gD})^*$ induces 
 isomorphisms
(also denoted $f_t$) 
$$f_t:\wedge^i (\fu_-'' \oplus \fu') \to \wedge^i(\fg_t/\fl_{\gD})^*,$$
$$f_t:C^i(\fr) \to C^i(\fg_t, \fl_{\gD}).$$
Let $d_t$ be the differential on $C^{\cdot}(\fr) := \oplus C^i(\fr)$ defined by
\begin{equation} \label{e.psit}
d_t = f_t^{-1} \circ d_{\fg_t} \circ f_t.
\end{equation}
Then 
\begin{equation}\label{e.identifycomplexes}
H^i(C^{\cdot}(\fr), d_t) \cong H^i(\fg_t, \fl_{\gD})
\end{equation}
for all $i$ and $t$.
In particular,
$H^*(C^{\cdot}(\fr), d_t)$ is isomorphic to
$H^*(\fg, \fl)$ for $t = 1$ and to
$H^*(\fr^{op})^L \simeq H^*(\fu \oplus \fu_-)^L$ for $t = 0$.

The space $\bigwedge^i \fr$ is equipped with a 
differential $\pa: \bigwedge^i \fr \to \bigwedge^{i-1} \fr$
corresponding to the Lie algebra homology of $\fr$ \cite[I.2.5]{BoWa:00}.
The map $\pa$ is $L$-equivariant, so it induces a 
map $\pa: C^i(\fr) \to C^{i-1}(\fr)$.   Then
$H_i(C^{\cdot}(\fr), \pa) = H_i(\fr)^L \simeq H_i(\fu \oplus \fu_-)^L$.

For $s\in Z$, recall the automorphism $\gG_s: \fr \to \fr$
from equation \eqref{e.gammas}. Extend $\gG_s$ to an automorphism
of $\bigwedge \fr$
so that
$$
\gG_s (X_1 \wedge \cdots \wedge X_k) = \gG_s (X_1) \wedge \cdots \wedge \gG_s( X_k) .
$$

\begin{Rem}\label{r.gammasaction} The operator $\gG_s$ preserves the subspace $C^{\cdot}(\fr)$ of
$\bigwedge \fr$.
Since $\gG_s$ is a Lie algebra automorphism,
it commutes with $\pa$.
\end{Rem}

For $s\in Z$,
the Lie algebra isomorphism
$(s, s^{-1}): \fg_t \to \fg_{s t}$ preserves
$\fl_{\gD}$.  Hence the dual 
$$
(s,s^{-1})^*: C^{\cdot}(\fg_{s t}, \fl_{\gD}) \to C^{\cdot}(\fg_{t}, \fl_{\gD})
$$
is a map of cochain complexes, i.e., 
$(s,s^{-1})^* \circ d_{\fg_{s t}} = d_{\fg_t} \circ (s,s^{-1})^*$.
In other words,
\begin{equation} \label{e.differentials}
d_{\fg_{s t}} = (s^{-1}, s)^* \circ d_{\fg_t} \circ (s, s^{-1})^*.
\end{equation}
Further, by functoriality of the exterior algebra and Lemma \ref{l.fgamma},

\begin{Lem} \label{l.fgammacomplex}
$$
(s,s^{-1})^* \circ f_{s t} = f_t \circ \gG_{s^{-1}}.
$$
\end{Lem}

\begin{Prop} \label{p.differentials}
Let $s \in Z$ and $t \in \C^m$.  Then
\begin{equation} \label{e.differentialsprop}
d_{s t} = \gG_{s} \circ d_t \circ \gG_{s^{-1}}.
\end{equation}
\end{Prop}

\begin{proof}
By definition,
$$
d_{s t} = f_{s t}^{-1} \circ d_{\fg_{s t}} \circ f_{s t}.
$$
By \eqref{e.differentials}, this equals
$$
f_{s t}^{-1} \circ  (s^{-1}, s)^* \circ d_{\fg_t} \circ (s, s^{-1})^*   \circ f_{s t}.
$$
By Lemma \ref{l.fgammacomplex}, this equals
$$
\gG_s \circ f_t^{-1} \circ  d_{\fg_t} \circ f_t \circ \gG_{s^{-1}}.
$$
By definition, $d_t = f_t^{-1} \circ  d_{\fg_t} \circ f_t$.  The result
follows.
\end{proof}

The following corollary will be used in the next section.

\begin{Cor} \label{c.cohisom}
For $s \in Z$ and $t \in \C^m$, the automorphism $\gG_s$
of $C^{\cdot}(\fr)$ induces a map of cochain complexes
$(C^{\cdot}(\fr), d_t)
\to (C^{\cdot}(\fr), d_{s t})$ .  Hence
$\gG_s$ induces an isomorphism in cohomology:
$$
\gG_s: H^{*}(C^{\cdot}(\fr), d_t) \stackrel{\simeq}{\rightarrow}
H^{*}(C^{\cdot}(\fr), d_{s t}).
$$
\end{Cor}

\begin{proof}
We can rewrite \eqref{e.differentialsprop} as
$$
 d_{s t} \circ \gG_s = \gG_s \circ d_t,
 $$
proving the first statement.  The second statement follows from
the first since
$\gG_s$ is a vector space isomorphism.
\end{proof}

We recall a result of Kostant \cite[Theorem 5.14]{Kos:61}.

\begin{Thm} \label{t.kostant61}

\noindent (1) $H^i(\fu \oplus \fu_-)^L = 0$ for $i$ odd.

\noindent (2) $\dim H^i(\fu \oplus \fu_-)^L = |\{ w\in W^P : 2 l(w) = i \}|$
if $i$ is even.
\end{Thm}

Choose a maximal compact subgroup  $K \subset G$ so that $K$ is the fixed subgroup
of a Cartan involution $\theta$ of $G$ which stabilizes $H$, $Z$, and $L$. 
Let $K_L = K\cap L$. Then $K/K_L = G/P$ and $H^*(\fk, \fk_L) \cong H^*(\fg, \fl)$
since we are computing relative Lie algebra cohomology with complex coefficients. It
follows that
$$
H^*(G/P) \cong H^*(K/K_L) \cong H^*(\fk, \fk_L) \cong H^*(\fg, \fl),
$$
using standard results on the cohomology of compact homogeneous spaces.
Thus, by Theorem \ref{t.kostant61} and the description of $H^*(G/P)$ in
Section \ref{s.schubert},
\begin{equation} \label{e.coh1iscoh0}
\dim(H^i(G/P)) = \dim(H^i(\fg_1, \fl_{\gD})) = \dim(H^i(\fg_0, \fl_{\gD}))
\end{equation}
for all $i$.

We consider the map $d: \C^m \to \End(C^{\cdot}(\fr))$ given by
$d(t)=d_t$.

\begin{Lem} \label{l.ccvb}
The map $d$ is a morphism of algebraic varieties.
\end{Lem}

\begin{proof}
It suffices to show that $d:\C^m \to \End(\bigwedge \fu_{-}'' \oplus \fu')$
is a morphism, since $C^{\cdot}(\fr)$ is a $d_t$-stable subspace
of $\bigwedge \fu_{-}'' \oplus \fu'$. Since $d_t$ is a derivation,
it suffices to show that $d_t: \fu_{-}'' \oplus \fu' \to \bigwedge^{2}
 (\fu_{-}'' \oplus \fu')$ is a morphism.  By definition of $d_t$, 
\begin{equation}\label{e.dformula}
d_t \phi_{\ga} (E_{\gb}(t), E_{\gg}(t)) = -\phi_{\ga}[E_{\gb}(t), E_{\gg}(t)],
\ \forall  \ \ga, \gb, \gg \in R(\fu \oplus \fu_-), 
\end{equation}
where $R(\fu \oplus \fu_-)=R^+(\fu) \cup -R^+(\fu)$.
Let $\gb, \gg \in R(\fu \oplus \fu^-)$ and if $\gb + \gg \in R(\fu \oplus \fu_-),$
we define $c_{\gb, \gg}$ by the rule $[e_{\gb}, e_{\gg}]= c_{\gb, \gg} e_{\gb + \gg}$.
If $\gb + \gg$ is not in $R(\fu \oplus \fu^-)$, we set $c_{\gb, \gg} E_{\gb + \gg}(t)
= 0$. 
 Then it follows from the definition
of $E_{\pm \ga}(t)$ that modulo $\fl_{\gD}$, 

\par\noindent $[E_{\gb}(t), E_{\gg}(t)] = c_{\gb, \gg} E_{\gb + \gg}(t)$ if
$\gb, \gg \in R(\fu \oplus \fu_-)$ have the same sign;
\par\noindent $[E_{\gb}(t), E_{-\gg}(t)]=t_{\gg}^2 c_{\gb, -\gg} E_{\gb - \gg}(t)$ if
$\gb, \gg, \gb - \gg \in R^+(\fu)$;
\par\noindent $[E_{\gb}(t), E_{-\gg}(t)]=t_{\gb}^2 c_{\gb, -\gg} E_{-(\gb - \gg)}(t)$
if $\gb, \gg, \gg - \gb \in R^+(\fu)$.

\par\noindent Applying these identities to equation \eqref{e.dformula} proves
the Lemma.
\end{proof}

As a consequence of Lemma \ref{l.ccvb}, 
$\cz^i := \ker(d|_{C^i(\fr)})$ and $\cb^i := \im(d|_{C^{i-1}(\fr)})$
are coherent subsheaves of the sheaf corresponding to the trivial vector bundle
$C^i(\fr)$ over $\C^m$, and standard properties of sheaf cohomology imply
the following remark.

\begin{Rem} \label{r.sheafhi}
 $\ch^i := \cz^i / \cb^i$
is a coherent sheaf and $\ch_t^i \cong H^i(\fg_t, \fl_{\gD})$.
\end{Rem}

\begin{Thm} \label{t.rkhconstant}
For $t\in \C^m$, $\dim(H^i(\fg_t, \fl_{\gD})) = \dim(H^i(G/P))$.
\end{Thm}

\begin{proof} 
Let $d_t^i$ denote the restriction of $d_t$ to $C^i(\fr)$, and let
$$
M_k^i = \{ t \in \C^m \ | \ \rank d_t^i \le k \}.
$$
Then $M_k^i$ is closed.  Indeed, if we choose bases of $C^i(\fr)$ and
$C^{i+1}(\fr)$, then $d_t^i$ corresponds to a matrix with polynomial entries
and $M_k^i$ is defined by the vanishing of the $(k + 1) \times (k + 1)$ minors of this
matrix.  Also, $M_k^i$ is $Z$-invariant, since Proposition \ref{p.differentials}
implies that if $s \in Z$, $t \in \C^m$, then $\rank d_t^i = \rank d_{st}^i$.

Suppose that $\co_1 = Z \cdot t_1$ and $\co_2 = Z \cdot t_2$ are $Z$-orbits
on $\C^m$ with $\co_1 \subseteq \overline{\co}_2$.  Since $M_k^i$ is closed
and $Z$-invariant, if $t_2 \in M_k^i$ then $t_1 \in M_k^i$. 
It follows easily that $\rank d_{t_1}^i \le \rank d_{t_2}^i$ for all $i$. 
 Therefore,
$$
\dim (\ker(d_{t_1}^i)) \ge \dim (\ker(d_{t_2}^i)).
$$
Similarly,
$$
\dim (\im (d_{t_1}^{i-1})) \le \dim (\im (d_{t_2}^{i-1})).
$$
Hence, 
\begin{equation} \label{e.lowerbound}
\dim(H^i(C^\cdot(\fr), d_{t_2})) \le \dim(H^i(C^\cdot(\fr), d_{t_1})).
\end{equation} 

The orbit $Z \cdot 1 = \{ s : s\in Z \}$ is open and dense in
$\C^m$, and the orbit $\{ 0 \} = Z \cdot 0$ is contained in the closure
of any $Z$-orbit on $\C^m$.  Therefore, equation \eqref{e.lowerbound}
implies that for all $i$ and for all $t \in \C^m$,
$$
\dim(H^i(C^\cdot(\fr), d_1)) \le \dim(H^i(C^\cdot(\fr), d_t)) \le
\dim(H^i(C^\cdot(\fr), d_0)).
$$
Since 
$$
\dim(H^i(C^\cdot(\fr), d_1)) = \dim(H^i(C^\cdot(\fr), d_0)) = \dim H^i(G/P)
$$
by equation \eqref{e.coh1iscoh0}, we see that $\dim(H^i(C^\cdot(\fr), d_t)) =
 \dim H^i(G/P)$, as desired.
 \end{proof}

\begin{Rem}\label{r.kerdtcon}
It follows from the proof that $\dim(\ker(d_t))$ is constant for $t\in \C^m$.
\end{Rem}

We let 
\begin{equation}\label{e.totalcoh}
\chstar := \oplus_{i=0}^{2N} \ch^i
\end{equation}
denote the corresponding total cohomology
sheaf, and note that $\chstar$ is a sheaf of rings.

\begin{Cor} \label{c.csheafrank}
 The coherent sheaf $\ch^i$ on $\C^m$ is a vector bundle. The rank
of the bundle $\chstar$ is 
$|W^P|$.
\end{Cor}

\begin{proof}
The first claim follows since a coherent sheaf with constant fiber dimension is
a vector bundle (\cite[Exercise II.5.8(c)]{Har:77}). The second claim follows
from Theorem \ref{t.rkhconstant}.
\end{proof}

\section{Global Schubert classes} \label{s.globalsch}
In the previous section, we showed that for every $t\in \C^m$, the vector
spaces $H^*(G/P)$ and  $H^*(\fg_t, \fl_{\gD})$
have the same dimension.  In this section, we use that result to give an
explicit isomorphism $H^*(G/P) \to H^*(\fg_t, \fl_{\gD})$, so that
the rings $H^*(\fg_t, \fl_{\gD})$ give a family of ring structures on
$H^*(G/P)$. More precisely, we identify a basis of sections of
the bundle $\chstar$ indexed by $w\in W^P$, and show that the structure
constants in this basis give the Belkale-Kumar family of cup products
on $H^*(G/P)$.

\subsection{Global sections of the cohomology bundle} \label{s.globalcoh}
Because we work with specific bases, we must introduce certain normalizations.
The decomposition $\fr = \fu_{-}'' \oplus \fu'$ induces an
identification $C^{\cdot}(\fr)$ with $(\bigwedge \fu_{-}'' \otimes \bigwedge \fu')^L$.
 We choose our root vectors so that
$$
(e_{\ga}, e_{-\ga}) = 1,
$$
where the inner product is the Killing form on $\fg$.

\begin{Def} \label{def.ef}
Fix an enumeration of the elements of $R^+(\fu)$.  For a subset  
$B = \{ \gb_1, \ldots, \gb_k \} \subset R^+(\fu)$, relabel the $\gb_j$
so that if $i<j$, then $\gb_i$ occurs before $\gb_j$ in our enumeration,
and define

\begin{equation} \label{e.defe1}
e'(B) = e_{\gb_1}' \wedge \cdots \wedge e_{\gb_{k}}' \mbox{  and  }
e''(B) =e_{-\gb_1}'' \wedge \cdots \wedge e_{-\gb_{k}}''.
\end{equation}
Define
\begin{equation} \label{e.defeb}
e(B_1, B_2) = e''(B_1) \wedge e'(B_2).
\end{equation}
If $B_1 = B_2 = B$ we write simply $e(B)$ for $e(B,B)$.   Recall the
monomials $t_{\ga}$ from Section \ref{s.subfamily} and
define a polynomial function $F_{B_1,B_2}$ on $\C^m$ by the formula
\begin{equation} \label{e.defFB}
F_{B_1,B_2}(t) = \prod_{\ga \in B_1} t_{\ga} \prod_{\gb \in B_2} t_{\gb}.
\end{equation}
If $B_1 = B_2 = B$ we write 
$F_B$ for $F_{B_1, B_2}$.
For  $ w \in W^P$, we note that $R^+ \cap w^{-1} R^- \subset R^+(\fu)$, and we  define
\begin{equation} \label{e.defew}
e(w) = e(R^+ \cap w^{-1} R^-)
\end{equation}
and consider the regular function on $\C^m$ given by
\begin{equation} \label{e.defFw}
F_w(t) := F_{R^+ \cap w^{-1} R^-}(t) =   
\prod_{\ga \in R^+ \cap w^{-1} R^-} t_{\ga}^2.
\end{equation}
\end{Def}

Recall that for $s \in Z$ we have defined $\gG_s: C^{\cdot}(\fr) \to
 C^{\cdot}(\fr)$.
\begin{Lem} \label{l.gammafb}
$\gG_s (e(B_1,B_2)) = F_{B_1,B_2}(s) e(B_1,B_2)$.
\end{Lem}

\begin{proof}
This follows immediately from the definition of $\gG_s$.
\end{proof}

Define a positive definite Hermitian form $\{ \cdot, \cdot \}$
 on $\bigwedge \fu_{-}'' \otimes \bigwedge \fu'$
as in \cite[Section 2.4]{Kos:63}, so that the elements $e(B_1,B_2)$ give an orthonormal
basis of $\bigwedge \fu_{-}'' \otimes \bigwedge \fu'.$ 
 Let $\pastar$ denote the Hermitian adjoint
of $\pa$, and let $L_\fr = \pa \pastar + \pastar \pa$ be the corresponding
Laplacian. Let $L_0$ be the Green's operator for
$L_\fr$, so by definition,
$$
L_0(\ker(L_\fr))=0, \ \text{ and } \ L_0 \circ L_\fr (x) = L_\fr \circ L_0 (x) = x, \ x\in \im(L_\fr).
$$
By \cite[Theorem 5.7]{Kos:61}, $L_\fr$ acts as a constant multiple of the identity
on each $H\times H$
weight space of $C^{\cdot}(\fr)$.  
It follows that $L_0(e(B_1,B_2))$ is a multiple (possibly $0$) 
of $e(B_1,B_2)$.
An explicit formula for $L_0$ may be found in \cite[5.6.9]{Kos:63}.

Define an operator $E$ on $\bigwedge \fu_{-}'' \otimes \bigwedge \fu'$ by
\begin{equation} \label{e.defEI}
E = 2 \sum_{\ga \in R^+(\fu)} \ad (e_{-\ga}'') \otimes \ad (e_{\ga}')
\end{equation}
Let $R = -L_0 E$. Since $E$ is strictly upper triangular and  $L_0$ is diagonal
with respect to the basis $\{ e(B_1, B_2) \}$ of $(\bigwedge \fu_{-}'' \otimes \bigwedge \fu')$,
and $E$ and $L_0$ are $L$-equivariant, it follows that $R$ is a nilpotent operator
on $C^{\cdot}(\fr)$.

\begin{Def} \label{d.procedure}
Given subsets $B_1, B_2, B_3, B_4$ of subsets of $R^+(\fu)$,
we say $(B_3, B_4) > (B_1, B_2)$ if $F_{B_1, B_2}(t)$ divides
$F_{B_3, B_4}(t)$ in $\C[t_1, \dots, t_m]$ and $\frac{F_{B_3, B_4}}{F_{B_1, B_2}}(t)$
vanishes at $0$. We say $(B_3, B_4) > B_1$ if $(B_3, B_4) > (B_1, B_1)$.
\end{Def}

\begin{Rem} \label{r.procedure}
If $B_1, \dots, B_6$ are subsets of $R^+(\fu)$, and if
$(B_5, B_6) > (B_3, B_4) > (B_1, B_2)$, then it is easy to check
that $(B_5, B_6) > (B_1, B_2)$.
\end{Rem}

\begin{Def} \label{d.wtpoly}
For $B_i, B_j$ subsets of $R^+(\fu)$, let
$$
C^k(\fr)_{B_i, B_j} = \{ v\in C^k(\fr) : \Gamma_s\cdot v = F_{B_i,B_j}(s) v, \ \forall
s \in Z \}.
$$
For a subset $B$ of $R^+(\fu)$, let $C^k(\fr)_B =C^k(\fr)_{B,B}$, and
$C^k(\fr)_w = C^k(\fr)_{R^+ \cap w^{-1}(R^-)}$.
\end{Def}

\begin{Lem} \label{l.lemr}
If $B_1, \dots, B_4$ are subsets of $R^+(\fu)$ and $i > 0$, then for
$v\in C^k(\fr)_{B_1, B_2}$, then $R^i(v) \in C^k(\fr)_{B_3, B_4}$
with $(B_3, B_4) > (B_1, B_2)$.
\end{Lem}

\begin{proof}
By Remark \ref{r.procedure}, it suffices to verify the assertion when 
$i=1$. Since $L_0$ acts diagonally on weight spaces, it suffices to
prove that for $\ga \in R^+$, 
$$\ad(e_{-\ga}'') \otimes \ad(e_{\ga})
(C^k(\fr)_{B_1, B_2}) \subset \oplus C^k(\fr)_{B_3, B_4},$$ 
where the sum is over $(B_3, B_4) > (B_1, B_2)$.
This follows from the fact that $\ad(e_{\ga})$ and
 $\ad(e_{-\ga})$ are derivations
of $\bigwedge \fu$ and $\bigwedge \fu_-$, respectively.
\end{proof}

In a series of papers (\cite{Kos:61}, \cite{Kos:63} and \cite{KoKu:86}), Kostant and 
Kostant-Kumar determined $d_1$-closed 
vectors $s_w \in C^{2l(w)}(\fr)$
with the property that the cohomology class $[s_w] \in H^{2l(w)}(\fr, d_1)$
corresponds to $\gre_w$ using the isomorphisms $H^*(C^{\cdot}(\fr), d_1) \cong
H^*(\fg_1, \fl_{\gD}) \cong H^*(G/P)$. We state their results in a form suitable
for our use.  We let $\rho = \frac{1}{2} \sum_{\ga \in R^+} \ga$ and 
define
\begin{equation} \label{e.deflw}
\gl_w = \prod_{\ga \in R^+ \cap w^{-1} R^- } \frac{2\pi }{<\rho, \ga >},
\end{equation}
where $< \cdot, \cdot >$ is the form on $\fh^*$ induced by the Killing form.
For a rational $H$-module $M$ and a character $\chi$ of $H$, let $M_\chi$ denote
the $\chi$-weight space of $M$, and for $v\in M$, let $v_\chi$ denote the
projection of $v$ to $M_\chi$ with respect to the decomposition, $M = \oplus M_\nu$,
where $\nu$ runs over characters of $H$.

\begin{Thm} \label{t.kostantbwpaper} (\cite{Kos:61}, \cite[Theorem 42]{BeKu:06})
For each $w\in W^P$, there exists a unique vector $k^w$ up to scalar multiplication such that
$$
\C k^w \in (\wedge^{l(w)}(\fu_-'' ) \otimes \wedge^{l(w)}(\fu'))^L
$$
and
\par\noindent (1) $k^w$ is $d_0$-closed;
\par\noindent (2) For $s\in Z$, $\gG_s (k^w) = F_w(s) k^w.$
\par\noindent (3)  $H^*(C^{\cdot}(\fr), d_0)$ has basis $\{ k^w : w\in W^P \}$.
\end{Thm}

We let
\begin{equation} \label{e.defsw}
s_w = (1-R)^{-1} k^w 
=  k^w + R k^w + R^2 k^w + \cdots.
\end{equation}

\begin{Thm} \label{t.koskoskumar} \cite{Kos:63}
For $w\in W^P$, then $s_w$ is $d_1$-closed, and we can (and do) normalize $k^w$
so that
 $f_1(s_w)$ is a representative for
the cohomology class $\gre_w \in H^{2l(w)}(G/P)$, using the
identification $H^*(\fg_1, \fl_{\gD})\cong H^*(G/P)$ from
equation \eqref{e.coh1iscoh0}.
\end{Thm}

\begin{Rem} \label{t.kostexplain}
Explicitly let $e''(w)=e''(R^+ \cap w^{-1}R^-)$ and let $e'(w) = e'(R^+ \cap w^{-1}R^-)$.
Then $M_w := U(\fl)\cdot e''(w)$ is an irreducible representation of $L$ in
$\bigwedge^{l(w)}(\fu_-'')$ with highest weight $w^{-1} \rho - \rho$. 
Further, $N_w := U(\fl) \cdot e'(w)$ is an irreducible representation of 
$L$ with lowest weight $\rho - w^{-1} \rho$ and $N_w$ is isomorphic to the dual
$M_w^*$ of $M_w$. Thus, $(M_w \otimes N_w)^L \cong \Hom_L(M_w, M_w)$ is
one-dimensional, and $k^w$ is a nonzero vector in $(M_w \otimes N_w)^L$.
If we normalize $k^w$ so its projection to $(M_w)_{w^{-1}(\rho) - \rho} \otimes N_w$
is in $\frac{i^{l(w)^2} e''(w)}{\gl_w} \otimes N_w$, this agrees with the normalization
given in Theorem \ref{t.koskoskumar} (see \cite[Theorem 43]{BeKu:06}).
\end{Rem}



Since $s_w$ is a $d_1$-cocycle, $f_1(s_w)$ is a $d_{\fg_1}$-cocycle.
For $s \in Z$, we
define a cocycle $\cs_w(s) \in (C^{2l(w)}(\fr), d_{s})$ by
\begin{equation} \label{e.defswt}
\cs_w(s) = \gG_s (s_w).
\end{equation}
We view $\cs_w$ as a function from $Z$ to $C^{\cdot}(\fr)$.
We show below that $\cs_w$ extends from $Z$ to all of 
$\C^m$.  
By Corollary \ref{c.cohisom}, for each $s \in Z$, $\cs_w(s)$ 
is a $d_{s}$-cocycle,
and by Theorem \ref{t.koskoskumar},
the classes $[\cs_w(s)]$ give a basis of $ H^*(C^{\cdot}(\fr), d_{s})$.
Note that by Lemma \ref{l.fgammacomplex},
$$
f_{s}(\cs_w(s)) = (s^{-1}, s)^* f_1(s_w).
$$

\begin{Lem} \label{l.gammasw}
Let $w \in W$ and $s \in Z$.  Then
\begin{equation} \label{e.gammasw}
\cs_w(s) =  F_{w}(s) k^w
+ \sum_{B_1, B_2} c_{B_1,B_2} F_{B_1,B_2}(s) e(B_1,B_2),
\end{equation}
where $c_{B_1,B_2} \in \C$,
and the sum is over all 
pairs of subsets $B_1,B_2$ of $R^+(\fu)$ such that 
$(B_1, B_2) > R^+ \cap w^{-1}R^-$. 
In particular, for $t\in \C^m$, each $F_{B_1,B_2}(t)$ is divisible by $F_w(t)$, and
$F_{B_1,B_2}/F_w$ defines a regular function on $\C^m$ vanishing at $0$.
\end{Lem}

\begin{proof}
By definition, $s_w = k^w + \sum_{i > 0} R^i k^w,$ so
$\cs_w(s) = \gG_s k^w + \sum_{i > 0} \gG_s R^i k^w.$
By Theorem \ref{t.kostantbwpaper}, 
$k^w \in C^{2l(w)}(\fr)_{R^+ \cap w^{-1} R^-},$
so by Lemma \ref{l.lemr}, 
$R^i k^w = \sum_{B_1, B_2} c_{B_1, B_2} e(B_1, B_2)$, where the sum is
over pairs $(B_1, B_2)$ such that $(B_1, B_2) > R^+ \cap w^{-1} R^-$.
equation \eqref{e.gammasw} now follows from the formula for the action of 
$\gG_s$ on $e(B_1,B_2)$ from
 Lemma \ref{l.gammafb}, and the remainder follows using
Definition \ref{d.procedure}.
\end{proof}

We may now define $\cs_w(t)$ for $t\in \C^m$ using equation \eqref{e.gammasw}
with $t$ in place of $s$.
The lemma implies that $\cs_w$ is a regular function
on $\C^m$ (with values in $C^{\cdot}(\fr)$).

\begin{Def} \label{def.global}
Define  $G_w(t) \in C^{\cdot}(\fr)$, for $t \in \C^m$,
by the formula
$$
G_w(t) =  k^w
+ \sum_{B_1,B_2} c_{B_1,B_2} \left( \frac{F_{B_1,B_2}}{F_w} \right) (t) 
e(B_1,B_2),
$$
where $c_{B_1,B_2}$ are as in Lemma \ref{l.gammasw}.
(Note that by Lemma \ref{l.gammasw}, $\frac{F_{B_1,B_2}}{F_w} $
is a regular function on $\C^m$.)
\end{Def}

If $s \in Z$, then $F_w(s) \neq 0$, and then
\begin{equation}\label{e.gwdef}
G_w(s) = \frac{\cs_w(s)}{F_w(s)}.
\end{equation}

\begin{Lem} \label{l.gammatlambdaw}
If $s \in Z$ and $t \in \C^m$, then
\begin{equation} \label{e.gtlw}
\gG_s (G_w(t)) = F_w(s) G_w(s t).
\end{equation}
\end{Lem}

\begin{proof}
By definition,
$$
F_w(s) G_w(s t) =   F_w(s) k^w
+ \sum_{B_1,B_2} c_{B_1,B_2} \left( \frac{F_{B_1,B_2}}{F_w} \right) (s t) 
F_w(s) e(B_1,B_2),
$$
On the other hand, 
by Lemma \ref{l.gammafb}, 
$$
\gG_s(G_w(t)) =  F_w(s) k^w
+ \sum_{B_1,B_2} c_{B_1,B_2} \left( \frac{F_{B_1,B_2}}{F_w} \right) (t) F_{B_1,B_2}(s) e(B_1,B_2).
$$
As
$$
 \left( \frac{F_{B_1,B_2}}{F_w} \right) (s t) F_w(s) = 
\left( \frac{F_{B_1,B_2}}{F_w} \right) (t) F_{B_1,B_2}(s),
 $$
 the lemma follows.
 \end{proof}

Recall the total cohomology bundle $\chstar$ 
on $\C^m$ from equation \eqref{e.totalcoh}.
As noted above, $\cs_w(s)$ is a $d_{s}$-cocycle
for $s \in Z$.  By continuity this implies that $\cs_w(t)$ is a $d_t$-cocycle
for all $t \in \C^m$.  Let $[\cs_w]$ denote the element
of $\chstar(\C^m)$ (i.e. the section
of $\chstar$ on $\C^m$) defined by $\cs_w$.  As noted above,
for
$s\in Z$, the classes $[\cs_w(s)]$ give
 a basis of $\chstar_{s}$.
However, this is not true for all $t \in \C^m$: in particular,
it fails at $t =0$, since
$\cs_w(0) = 0$ if $w \neq e$ by equation \eqref{e.gammasw}.   The next theorem
shows that by replacing $\cs_w$ by $G_w$ we obtain a basis
for all $t \in \C^m$.

\begin{Thm} \label{t.main1}
For all $t \in \C^m$, the classes $G_w(t)$ are elements of
$C^{\cdot}(\fr)$ which are $d_t$-cocycles, and $\{ [G_w(t)] \}_{w\in W^P}$
is a basis of $H^*(C^{\cdot}(\fr), d_t)$.
Thus, the class of each $G_w$ defines a global section
of the vector bundle $\chstar$ on $\C^m$, and
 the classes $[G_w]$ give a trivialization of this vector
bundle on $\C^m$.
\end{Thm}

\begin{proof}
For
$s\in Z$, $G_w(s)$ is a constant multiple of $\cs_w(s)$.
Since $\cs_w(s)$ is a $d_{s}$-cocycle, so is $G_w(s)$.  
By continuity, $G_w(t)$ is a $d_t$-cocycle for
all $t \in \C^m$.  
By the preceding paragraph, the class of each $G_w$
defines a global section of the vector bundle $\chstar$; we denote
this section by $[G_w]$.
Note that $G_w(0) =  k^w$.
It follows from Kostant's Theorem \ref{t.kostantbwpaper} that
the
$[G_w(0)]$ form a basis of the cohomology
group $\chstar_0 = 
H^{*}(C^{\cdot}(\fr), d_0)$.
Therefore there is an
open neighborhood $A$ of $0$ such that  the
$[G_w(a)]$ form a basis of $\chstar_a$ for $a \in A$.  
If $s\in Z$, then by Corollary \ref{c.cohisom},
 $\gG_s$ gives an isomorphism $\chstar_a
\to \chstar_{s a}$.  It follows that the
$[\gG_s(G_w(a)]$ form a basis for $\chstar_{s a}$.
But $[\gG_s(G_w(a))] = [F_w(s) G_w(s a)]$.  The scalar
$F_w(s)$ is nonzero as $s \in Z$, so the $[G_w(s a)]$
form a basis of $\chstar_{s a}$ for all $s a \in ZA$.
But $ZA = \C^m$, and  the theorem follows.
\end{proof}

Let $c_{uv}^w$ denote the structure constants in the ring
$H^{*}(\fg, \fl)$ with respect to the Schubert
basis $\gre_w := G_w(1)$.  In other words,
\begin{equation} \label{e.epsilonconstants}
\gre_u \gre_v = \sum_w c_{uv}^w \gre_w.
\end{equation}

\begin{Def} \label{d.global}
Write $\gre_w(t)$ for the section $[G_w]$ of $\chstar$ on
$\C^m$.  We refer to the $\gre_w(t)$ as global Schubert
classes.
\end{Def}

Hence, $\gre_w(1) = \gre_w \in \chstar_1 \cong H^*(G/P)$.

Theorem \ref{t.main1} implies that $\chstar(\C^m)$ 
is a free $\C[t_1, \dots, t_m]$-module with basis $\gre_w(t)$.  
As noted in Section \ref{s.rellie}, $\chstar(\C^m)$ is a ring.  Therefore
we can write
$$
\gre_u(t)\gre_v (t)= \sum_w C_{uv}^w (t)\gre_w(t),
$$
where each $C_{uv}^w(t)$ is a regular function on $\C^m$.

\begin{Thm} \label{t.mult}
Let $u,v$ and $w$ be in $W$. 
The multiplication in the ring $\chstar(\C^m)$ is given
by
\begin{equation} \label{e.mult}
\gre_u (t) \gre_v(t) = \sum_w c_{uv}^w \frac{F_w(t)}{F_u(t) F_v(t)}
\gre_w(t).
\end{equation}
In other words,
$$
C_{uv}^w(t) = \frac{F_w(t)}{F_u(t) F_v(t)} c_{uv}^w.
$$
\end{Thm}

\begin{proof}
Since the restriction map $\chstar(\C^m) \to \chstar(Z)$
is injective, it suffices to check that \eqref{e.mult}
holds when we replace the global sections $\gre_w$
by their restrictions to $Z$.
The map $\gG_s: H^*(C^{\cdot}(\fr), d_1) \to H^*(C^{\cdot}(\fr), d_{{s}})$
is a ring isomorphism.  By definition, if $s\in Z$, then
$\cs_w(s) = \gG_s(S_w)$.  Therefore, for all $s \in Z$,
the multiplication in $\chstar_s$ satisfies
\begin{equation} \label{e.multproof}
[\cs_u(s)][\cs_v(s)] = \sum_w c_{uv}^w [\cs_w(s)].
\end{equation}
Therefore, we have in the ring $\chstar(\C^m)$ the equation
$[\cs_u(t)][\cs_v(t)] = \sum_w c_{uv}^w [\cs_w(t)]$.
By definition, 
$$
\gre_w|_Z = \left[\frac{\cs_w|_Z}{F_w} \right]
$$
(and similarly for $\gre_u|_Z$, $\gre_u|_Z$).
Substituting in \eqref{e.multproof} yields the result.
\end{proof}

For arbitrary $u,v,w \in W^P$, $ \frac{F_w(t)}{F_u(t) F_v(t)}$
is a rational function on $\C^m$ which need not be regular
on all of $\C^m$.  However,  since the structure
constants are regular functions, we have the following corollary.

\begin{Cor} \label{c.globalfunction}
If $c_{uv}^w \neq 0$, then $ \frac{F_w(t)}{F_u(t) F_v(t)}$ is
a regular function on $\C^m$.
\end{Cor}

This corollary was originally proved by Belkale and Kumar
using geometric invariant theory.

\subsection{The Belkale-Kumar product}
In this section, we show that the product on $\chstar(\C^m)$
coincides with
the Belkale-Kumar deformed cup product after reindexing.

For $w\in W^P$, let $w^* = w_0 w w_{0,P} \in W^P$
(see Remark \ref{r.klm}).
Following Belkale and Kumar, we reindex by setting
 $\gL_w(t) = \gre_{w^*}(t)$ for $w\in W^P$, and we
let $\gL_w = \gL_w(1)$.
The multiplication defined by Belkale and Kumar on
$H^*(G/P)$ is as follows.   Recall that
$\{ \ga_1, \cdots, \ga_n \}$ denotes the simple
roots, ordered as in Section \ref{s.subfamily}
so that $\ga_1, \dots, \ga_m \in R^+(\fu)$ and $\ga_{m+1}, \dots, \ga_n \in R^+(\fl)$.
Let $\{x_1, \ldots, x_n \}$ be the dual basis
of $\fh^*$.  
Let $\tau_1, \ldots, \tau_m$ be indeterminates.
They define $\chi_w \in \fh^*$ by
\begin{equation} \label{e.chiwdef}
\chi_w = \sum_{\gb \in R^+(\fu) \cap w^{-1} R^+} \gb.
\end{equation}
For $u, v, w \in W^P$, define integers $d_{uv}^w$ by
$$ 
\gL_u \gL_v = \sum_{w\in W^P} d_{uv}^w \gL_w.
$$
Using equation \eqref{e.epsilonconstants}, we 
see $d_{uv}^w = c_{u^* v^*}^{w^*}$.
Belkale-Kumar define a $ \Z[\tau_1, \ldots, \tau_m]$-linear
product $\odot$ on the
ring $H^*(G/P) \otimes \Z[\tau_1, \ldots, \tau_m]$
by the rule
\begin{equation} \label{e.BK}
\gL_u \odot \gL_v = \sum_w (\prod_{i=1}^m \tau_i^{(\chi_w - (\chi_u + \chi_v))(x_i)})
d_{uv}^w \gL_w.
\end{equation}

For $w\in W^P$, let 
$$
\eta_w = \sum_{\gb \in R^+(\fu) \cap w^{-1} R^-} \gb.
$$

\begin{Lem} \label{l.chiveta}
For $w\in W^P$, $\chi_w = w_{0,P}(\eta_{w^*}).$
\end{Lem}

\begin{proof}
Since $w_{0,P}(R^+(\fu))=R^+(\fu)$, then $R^+(\fu) \cap {w^*}^{-1}(R^-)
=w_{0,P}^{-1}(R^+(\fu) \cap w^{-1}R^+)$. The lemma follows
easily.
\end{proof}

Note that if $\eta_w = \sum_{i=1}^n k_i \ga_i$, then $F_w(t) = \prod_{i=1}^m
t_{i}^{2k_i}.$

\begin{Lem} \label{l.BK}
Define a homomorphism
$$
\phi: \Z[\tau_1, \cdots, \tau_m] \to \C[\C^m]
$$
by the rule $\phi(\tau_i) = t_i^2$.  
Then
$$
\phi(\prod_{i=1}^m \tau_i^{\chi_w(x_i)}) = F_{w^*}(t).
$$
\end{Lem}

\begin{proof}
Let $w_{0,P}(\eta_{w^*}) = \chi_w = \sum n_i \ga_i$, so by Lemma \ref{l.chiveta},
$$
F_{w^*}(t) = \prod_{i=1}^m  t_{i}^{2n_i} = w_{0, P} \prod_{i=1}^m  t_{i}^{2n_i}
 = \phi(\prod_{i=1}^m \tau_i^{\chi_w(x_i)}),
$$
 since $w_{0,P}$ acts trivially on $Z$.
\end{proof}

The next theorem implies that after making the change
of variables $\tau_i = t_i^2$, the structure constants of
the Belkale-Kumar product
are the same as those for the product on $\chstar(\C^n)$.

\begin{Thm} \label{t.BKcomp}
Give $H^*(G/P) \otimes \Z[\tau_1, \ldots, \tau_m]$ the
Belkale-Kumar product $\odot$. 
Extend
 $\phi: \Z[\tau_1, \ldots, \tau_m] \to \C[t_1, \dots, t_m]$ 
to a map (also
denoted $\phi$)
$$
\phi: H^*(G/P) \otimes \Z[\tau_1, \ldots, \tau_m] \to \chstar(\C^m)
$$
by requiring that $\phi(\gL_w)=\gL_{w}(t)$ and $\phi(a x)=\phi(a) \phi(x)$
for $a\in \Z[\tau_1, \ldots, \tau_m]$ and $x\in H^*(G/P) \otimes \Z[\tau_1, \ldots, \tau_m]$.
Then 
\begin{equation} \label{e.homomorphism}
\phi( \prod_i \tau_i^{(\chi_w - (\chi_u + \chi_v))(x_i)})
d_{uv}^w = C_{u^* v^*}^{w^*}(t).
\end{equation}
In other words, $\phi$ is a ring homomorphism.
\end{Thm}

\begin{proof}
By Lemma \ref{l.BK},
$$
\phi ( ( \prod_i \tau_i^{(\chi_w - (\chi_u + \chi_v))(x_i)} )
d_{uv}^w = \frac{F_{w^*}(t)}{F_{u^*}(t) F_{v^*}(t)} \cdot c_{u^* v^*}^{w^*},
$$
and this equals $C_{u^* v^*}^{w^*}$ by Theorem \ref{t.mult}.  
It follows from the definitions that $\phi(\gL_u (t) \gL_v (t) )$ equals
$\phi(\gL_u(t)) \phi(\gL_v(t))$.
Since $\phi$ takes structure constants to structure
constants, it is a ring homomorphism.
\end{proof}

\begin{Rem} \label{r.qcoh} The Belkale-Kumar family of cup products is quite
different from quantum cohomology of the flag variety. Indeed, quantum cohomology
of the flag variety adds extra nonzero terms to
 the usual cup product, 
while the Belkale-Kumar family degenerates the usual cup product to a ring
where some nonzero structure constants may become zero.
\end{Rem}

\begin{Rem}\label{r.ringnextpaper}
Let $P \subset Q$ be parabolic subgroups of $G$, and consider the 
$Q/P$-fiber bundle $\pi:G/P \to G/Q$.  It is a well-known
result of Borel that if $I$ is the
ideal of $H^*(G/P)$ generated by $\sum_{i > 0} \pi^* H^i(G/Q)$,
then $H^*(Q/P) \cong H^*(G/P)/I$ \cite[Theorem 26.1]{Borel:53}. 
 In a future paper, we
plan to establish an analogous result for the Belkale-Kumar
cup product by using the Hochschild-Serre spectral sequence
to prove a Leray-Hirsch theorem for relative Lie algebra
cohomology.
\end{Rem}

\begin{Rem} \label{r.kmhint}
In a future paper, we plan 
 to study the Kac-Moody generalization of the Belkale-Kumar
cup product.  Although the arguments used in this paper
do not directly generalize, we will establish analogous
results in the Kac-Moody case using the Hochschild-Serre
spectral sequence.  
The Kac-Moody generalization of the Belkale-Kumar cup product
was 
used by Kumar in his proof of the Cachazo-Douglas-Seiberg-Witten conjecture
in the very interesting paper \cite{Ku:08}.
\end{Rem}

\begin{Rem} \label{r.producttK}
For $t=(t_1, \dots, t_m)$, let $J(t)=\{ 1 \le i \le m : t_i \not= 0 \}$.
Recall the subalgebras $\fg_{p_J}$ from Proposition \ref{p.iso3}.
By Remark \ref{r.gtisom} and Corollary \ref{c.cohisom}, 

\begin{equation} \label{e.fgtvsftpt}
\fg_t \cong \fg_{p_{J(t)}} \ {\rm and } \  H^*(\fg_t, \fl_{\gD}) \cong
H^*(\fg_{p_{J(t)}}, \fl_{\gD}).
\end{equation}
 Therefore, in order to understand the family of cup products,
it suffices to compute
$H^*(\fg_{p_J}, \fl_{\gD})$ for each $J \subset \{ 1, \dots, m \}$.
\end{Rem}

\section{Disjointness of $\pa$ and $d_t$} \label{s.disjoint}

The results of \cite{Kos:63} are based on the proof of the remarkable fact
that $\pa$ and $d_1$ are disjoint. This gives an identification
between $H_*(C^{\cdot}(\fr), \pa)$ and $H^*(C^{\cdot}(\fr), d_1)$,
and leads to the approach of Kostant and Kumar to Schubert calculus for
the flag variety of a symmetrizable Kac-Moody Lie algebra \cite{KoKu:86}.
In this section, we show that our methods imply that $\pa$ and $d_t$
are disjoint for all $t\in \C^m$.

\begin{Def} \label{d.disjointness} (\cite[2.1]{Kos:61})
Let $C^{\cdot}$ be a finite dimensional graded vector space
with linear maps $d:C^{\cdot} \to C^{\cdot + 1}$ and $\gd:C^{\cdot} \to C^{\cdot - 1}$
 of degree $1$ and $-1$ respectively. We say that
$d$ and $\gd$ are {\it disjoint} if 
$$
\im(d) \cap \ker(\gd) = \im(\gd) \cap \ker(d) = 0.
$$
\end{Def}

\begin{Prop} \label{p.disjointkos} \cite[Proposition 2.1]{Kos:61}
Let $d$ and $\gd$ be disjoint operators on the complex $C^{\cdot}$
and suppose $d^2={\gd}^2 = 0$. Let $S = d\gd + \gd d$.  Then
\par\noindent (1) $\ker(S) = \ker(d) \cap \ker(\gd)$.
\par\noindent (2) The natural maps $\ker(S) \to H^*(C^{\cdot}, d)$
and $\ker(S) \to H_*(C^{\cdot}, \gd)$ are isomorphisms.
\end{Prop}

We will give a proof of the following result at the end of this
section.

\begin{Thm} \label{t.disjointness}
The linear maps $d_t$ and $\pa$ of $C^{\cdot}(\fr)$ are disjoint
for all $t\in \C^m$.
\end{Thm}

We first establish the converse to Proposition \ref{p.disjointkos}.

\begin{Lem} \label{l.conversedisjoint}
 Let $C^{\cdot}$ be a finite dimensional graded vector
space with linear maps $d$ and $\gd$ of degree $+1$ and $-1$ respectively,
and suppose $d^2 = {\gd}^2 = 0$. Assume
\par\noindent (1) $\ker(S) \subset \ker(d)$ and $\ker(S) \subset \ker(\gd)$;
\par\noindent (2) The induced quotient maps $\psi^*:\ker(S) \to H^*(C^{\cdot}, d)$
and $\psi_*:\ker(S) \to H_*(C^{\cdot}, \gd)$ are isomorphisms.
\par\noindent Then $d$ and $\gd$ are disjoint.
\end{Lem}

\begin{proof}
Let $y \in \im(d) \cap \ker(\gd)$. Then
$S(y) = d \gd (y) + \gd d (y) = 0$. Thus, $y \in \ker(S)$, so since the cohomology class
$[y] \in H^*(C^{\cdot}, d)$ is $0$, it follows that $y=0$ from assumption (2).
This proves $\im(d) \cap \ker(\gd)=0$, and a similar argument shows that
$\im(\gd) \cap \ker(d) = 0$.
\end{proof}

\begin{Thm} \label{t.disjointkostant}(see \cite{Kos:63})
On $C^{\cdot}(\fr)$, the operators $d_1$ and $\pa$ are disjoint, and
the operators $d_0$ and $\pa$ are disjoint.
\end{Thm}

\begin{proof}
The case when $t=1$ is the assertion of \cite[Theorem 4.5]{Kos:63}.
For the case when $t=0$, 
let $\pastar$ denote the Hermitian adjoint of the operator $\pa$
on $C^{\cdot}(\fr)$ with respect to the 
the positive definite Hermitian form
$\{ \cdot, \cdot \}$ defined after Lemma \ref{l.gammafb}. 
 Because the form is positive definite,
$\pa$ and $\pastar$ are disjoint. 
Consider the symmetric bilinear form $( \cdot, \cdot )$ on $C^{\cdot}(\fr)$ given
by extending the Killing form on $\fr$, as in \cite[equation 3.2.1]{Kos:61}, and let
$\pa^{tr}$ be the transpose of $\pa$ with respect to $( \cdot, \cdot)$. Then
$\pastar=-\pa^{tr}$ by \cite[2.6.5]{Kos:63}. It follows from definitions that
 $\pa^{tr} = d_0$. Thus $d_0$ and $-\pa$ are disjoint, so $d_0$ and
$\pa$ are disjoint.
\end{proof}

\begin{Lem} \label{l.zdisjoint}
For $s\in Z$, the operators $d_s$ and $\pa$ are disjoint.
\end{Lem}

\begin{proof}
By Remark \ref{r.gammasaction}, $\pa \circ \gG_s = \gG_s \circ \pa$, 
and by Proposition \ref{p.differentials},
\begin{equation} \label{e.zdisjoint}
\gG_s:\ker(d_t)  \to \ker(d_st), \ \ \gG_s:\im(d_t)\to \im(d_{st}),
\end{equation}
is an isomorphism.  Applying these observations when $t=1$ yields
the result.
\end{proof}

For $t\in \C^m$, let $S_t:= d_t \pa + \pa d_t$.

\begin{Prop} \label{p.constantharmonic}
For $t\in \C^m$, $\dim(\ker(S_t)) = | W^P|$.
\end{Prop}

\begin{proof}
By Lemma \ref{l.zdisjoint} and Proposition \ref{p.disjointkos}, 
$\dim(\ker(S_s)) \cong H_*(C^{\cdot}, \pa)$ for
$s\in Z$.  By Theorem \ref{t.disjointkostant}, it follows
that $\dim(H_*(C^{\cdot}, \pa)) = \dim(\ker(S_0)) = \dim(H^*(C^{\cdot}, d_0))$,
and this last space has dimension $|W^P|$ by Theorem \ref{t.kostant61}.
Thus,  it follows that $\dim(\ker(S_t))=|W^P|$ for $t=0$ and for
$t\in Z$. Using families as in the proof of Theorem \ref{t.rkhconstant}, the
result follows.
\end{proof}

\begin{Prop} \label{p.harmonicclosed}
For all $t\in \C^m$, $\ker(S_t) \subset \ker(d_t)$ and $\ker(S_t) \subset
\ker(\pa)$.
\end{Prop}

\begin{proof}
Let $a = |W^P| \le b = \dim \ker d_0$.  Set
$$
M = \{ (V_1, V_2) \in \Gr(a, C^{\cdot}(\fr)) \times \Gr(b, C^{\cdot}(\fr)) \ | \ V_1 \subseteq V_2 \}.
$$
As $M$ fibers over $\Gr(a, C^{\cdot}(\fr))$ with fibers isomorphic to
$\Gr(b-a, \Gr(a, C^{\cdot}(\fr)))$, $M$ is complete, hence closed in
$ \Gr(a, C^{\cdot}(\fr)) \times \Gr(b, C^{\cdot}(\fr))$.

Consider the morphism
$$
\phi: \C^m \to  \Gr(a, C^{\cdot}(\fr)) \times \Gr(b, C^{\cdot}(\fr))
$$
defined by $\phi(t) = (\ker S_t, \ker d_t)$. (cf. Proposition
\ref{p.constantharmonic} and Remark \ref{r.kerdtcon}). For $s \in Z$, $\ker S_s \subset \ker d_s$
by Lemma \ref{l.zdisjoint} and Proposition \ref{p.disjointkos}, so $\phi(s) \in M$.
Thus, $\phi^{-1}(M) \supseteq Z$.  As $\phi^{-1}(M)$ is closed and $Z$ is dense in $\C^m$,
we see that $\phi^{-1}(M) = \C^m$.  Thus, for all $t \in \C^m$, $\phi(t) \in M$, so
$\ker S_t \subseteq \ker d_t$.

To show that $\ker(S_t) \subset
\ker(\pa)$, consider the morphism
$\rho: \C^m \to \Gr(a, C^{\cdot}(\fr))$ given by
$\rho(t) = \ker(S_t)$.  For $s \in Z$, $\ker(S_s) \subset \ker(\pa) \cap \ker(d_s)$ by Lemma \ref{l.zdisjoint} and Proposition \ref{p.disjointkos},
so $\rho(s)$ lies in the closed subset $\Gr(a, \ker(\pa))$ of $\Gr(a, C^{\cdot}(\fr))$.  Arguing
as in the preceding paragraph shows that for all $t \in \C^n$, $\rho(t)$ is in $\Gr(a, \ker(\pa))$,
so $\ker S_t \subset \ker (\pa)$.
\end{proof}


\noindent{\it Proof of Theorem \ref{t.disjointness}}:
For $t \in \C^m$, we have
quotient maps $\psi_t^* : \ker(S_t) \to H^*(C^{\cdot}(\fr), d_t)$
and $\psi_*^t:\ker(S_t) \to H_*(C^{\cdot}(\fr), \pa)$ 
(cf. Lemma \ref{l.conversedisjoint}).
Observe that each of the spaces $\ker S_t$, $H^*(C^{\cdot}(\fr), d_t)$, and $H_*(C^{\cdot}(\fr), \pa)$
has dimension equal to $|W^P|$ (see Proposition
\ref{p.constantharmonic} and its proof, and Theorem \ref{t.rkhconstant}),
so to show that either of the quotient maps is an isomorphism it is enough
to show injectivity or surjectivity.

We begin by showing that  $\psi_t^* : \ker(S_t) \to H^*(C^{\cdot}(\fr), d_t)$
is an isomorphism.
Observe that $G_w(t) \in \ker(\pa)$ for $t\in \C^m$.
To see this, note that by Remark \ref{r.gammasaction}, for $s\in Z$,
 $\gG_s:\ker(\pa)
\to \ker(\pa)$ is an isomorphism. By results from
\cite{Kos:63} (see  \cite[Theorem 5.6]{EvLu:99}), it follows that
 $s_w \in \ker(\pa)$. Hence 
$\cs_w(s) \in \ker(\pa)$ by equation \eqref{e.defswt} for $s\in Z$, so
$G_w(s) \in \ker(\pa)$ for $s\in Z$ by equation \eqref{e.gwdef}. By
continuity,
$G_w(t) \in \ker(\pa)$ for $t\in \C^m$ as desired. 

Since $G_w(t) \in \ker d_t$ by Theorem \ref{t.main1},
we see that $G_w(t) \in \ker(S_t)$.  The map $\psi_{t}^*$ takes
$G_w(t)$ to its $d_t$-cohomology class.  
By Theorem
\ref{t.main1}, these classes form a basis
of  $H^*(C^{\cdot}(\fr), d_t)$, so $\psi_t^* : \ker(S_t) \to H^*(C^{\cdot}(\fr), d_t)$
is surjective, hence an isomorphism.

Since
$\psi_t^*$ is an isomorphism,
Proposition \ref{p.disjointkos} and Lemma \ref{l.conversedisjoint}
imply that $\psi_*^t$ is an isomorphism if and only if $d_t$ and $\pa$
are disjoint; the set $A$ of $t \in \C^m$ for which this holds is $Z$-invariant
by \eqref{e.zdisjoint}.  To complete the proof we will show that $A = \C^m$.
Since $\ker \psi_*^t =  \ker S_t \cap \im \pa$,  
$\psi_*^t$ is an isomorphism
if and only if $\ker S_t \cap \im \pa = 0$.  As in the proof of
Proposition \ref{p.harmonicclosed}, let $a = |W^P|$ and consider 
the morphism
$\rho: \C^m \to \Gr(a, C^{\cdot}(\fr))$ given by $\rho(t) = \ker(S_t)$.
The subset
$$
N = \{ V \in \Gr(a, C^{\cdot}(\fr)) \ | \ V \cap \im \pa = \{ 0 \} \ \}
$$
is open, so
$A = \rho^{-1}(N)$ is open.
As $d_0$ and $\pa$ are disjoint, $0 \in A$,
so $A$ is an open $Z$-invariant set in $\C^m$ containing $\{ 0 \}$.
The only such set is $\C^m$ itself, so we conclude that
$A = \C^m$, as desired. \qed

\section{Appendix on generalized Levi movability}\label{s.levimov}

\begin{center}
\small\textsc{Sam Evens, William Graham, and Edward Richmond}
\end{center}
\vspace{.1in}

In \cite{BeKu:06}, the definition of the deformed product
$\odot_\tau$ is motivated by the geometric notion of Levi-movability.
Indeed, for $u,v, w\in W^P$, Theorem 15 in \cite{BeKu:06} asserts
that $\gL_u \odot_0 \gL_v$ has nonzero $\gL_w$ coefficient
in $H^*(G/P)$
if and only if related shifted Schubert cells can be made transverse
using the action of the Levi subgroup $L$ of $P$.  In this appendix,
we generalize this result to the product $\odot_\tau$ for
any value of $\tau$ by an argument following closely the proof
of Theorem 15 in \cite{BeKu:06}. We refer to \cite{BeKu:06} for
a more complete exposition of the needed background, and for
more detailed proofs.   We thank Shrawan Kumar for useful comments
that were helpful in proving the results in this section, and 
acknowledge that he also knew how to prove these results.

Recall our convention from Remark \ref{r.indexconvention} that 
$\ga_1, \dots, \ga_m$ denote the simple roots of $R^+(\fu)$.
For a subset $J$ of $\{ 1, \dots, m \}$,
we consider the subset $K=I\cup J$ of the set of simple roots of $\fg$,
and the Levi subalgebra and subgroup $\fl_K$ and $L_K$.
Following \cite{BeKu:06}, for $w\in W^P$, let $C_w = w^{-1}BwP \subset G/P$,
and note that $[\overline{C_w}]=[X_w]$ in $H_*(G/P)$.

\begin{Def} \label{d.lmovable}
For $s > 0$, we say a $s$-tuple $(w_1, \dots, w_s)$ of elements in $W^P$ is
$L_K$-movable if for some $s$-tuple $(l_1, \dots, l_s)$ in $L_K$, 
the intersection
$l_1 C_{w_1} \cap \dots \cap l_s C_{w_s}$ is transverse and nonempty
at a point  of $L_K\cdot eP$ in $G/P$.
\end{Def}

\begin{Rem} We give an equivalent definition of $L_K$-movability, which
is analogous to the definition in \cite{BeKu:06}, and we will use
this equivalent definition in the sequel.  We claim that
a $s$-tuple $(w_1, \dots, w_s)$ in $W^P$ is $L_K$-movable
if and only if there exists a $s$-tuple $(m_1, \dots, m_s)$ of elements
of $L_K \cap P$ such that $m_1 C_{w_1} \cap \dots \cap m_s C_{w_s}$
is transverse at the identity coset $eP$ of $G/P$ (see Definition 4 in
\cite{BeKu:06}).  Indeed, one direction of this equivalence is clear.
For the other direction, assume $l_1 C_{w_1} \cap \dots \cap l_s C_{w_s}$
has nonempty transverse intersection at $lP$, with $l, l_1,\dots, l_k \in
L_K$. Then $l^{-1}l_1 C_{w_1} \cap \dots \cap l^{-1}l_s C_{w_s}$
has nonempty transverse intersection at $eP$, and each $l^{-1} l_i \in
L_K\cap P$ by Lemma 1, p. 190, in \cite{BeKu:06}.
\end{Rem}

Let $\fb_L$ be the Borel subalgebra of $\fl$ containing the root
space $\fg_\ga$ for each root $\ga \in R^+(\fl)$, and let $B_L$ be
the corresponding Borel subgroup of $L$.  
For each $w_j \in W^P$, recall the integral weight $\chi_{w_j}$ from
equation \eqref{e.chiwdef}, and let $\C_{-\chi_{w_j}}$ be the corresponding
dual representation of the Cartan subgroup $H$. Let ${\cl}(w_j)$ denote
the induced line bundle $P\times_{B_L} \C_{-\chi_{w_j}}$ over $P/B_L$.
We denote by $1$ the identity element of $W$. We let

$$
\cl := (\cl(w_1) \otimes \cl(1)^{-1}) \boxtimes \cl(w_2) \dots \boxtimes \cl(w_s)
$$
denote the corresponding external tensor product of line bundles over
$(P/B_L)^s$. In Section 3 of \cite{BeKu:06} a section $\theta$ of
$\cl$ is constructed which is $P$-invariant for the diagonal action of
$P$ on $(P/B_L)^s$. Let $T=T_{eP}(G/P)$ and let $T_{w_j}$ denote
$T_e(C_{w_j})$. Assume that
\begin{equation} \label{e.codimcondition}
\sum_{j=1}^s \codim(C_{w_j}) = \dim(G/P).
\end{equation}
The section $\theta$ vanishes at a point $p=(p_1 B_L,
\dots, p_s B_L)$ of $(P/B_L)^s$ if and only if the projection
$T\to \oplus_{j=1}^s T/p_j T_{w_j}$ is surjective 
(cf. \cite{BeKu:06}, Lemma 7).

We recall some constructions from geometric invariant theory
(see Section 4 of \cite{BeKu:06}).
For an algebraic group $A$, we let $X_*(A)$ denote the group of cocharacters
of $A$.  For $\chi \in X_*(G)$, we let
\begin{equation} \label{e.stabilityparabolic}
P(\chi) = \{ g\in G : \lim_{t\to 0 } \chi(t) g \chi(t)^{-1} \ {\rm exists \ in \ 
G } \ \}
\end{equation}
be the associated parabolic subgroup of $G$.
We say that $\chi \in X_*(P)$ is $P$-admissible if $\lim_{t\to 0} \chi(t) x$ exists
in $P/B_L$ for all $x\in P/B_L$. We denote by $z_\rho \in \fz(\fl_K)$
the element of the center of $\fl_K$ such that $\alpha_i(z_\rho)=1$ for
all $i \in \{ 1, \dots, m \} - J$. Since $G$ is of adjoint type,
there is a unique one parameter subgroup $\gl_\rho$ of $Z(L_K)$ 
such that $d\gl_\rho (1)=z_\rho$.
By Lemma 12 of \cite{BeKu:06}, it follows that $\gl_\rho$ is
$P$-admissible.
 For an algebraic group $A$ 
with $A$-equivariant line bundle $\cm$ on a $A$-variety $Y$ with
$\chi \in X_*(A)$ and $y\in Y$, note that if the limit
 $y_0:=\lim_{t\to 0} \chi(t)y$ exists, then $y_0$ is
a fixed point for the $\Cstar$-action induced by $\chi$ on $Y$, and $\Cstar$
acts via $\chi$ on the fiber $\cm_{y_0}$. We let $\mu^{\cm}(y, \chi)=n$ if
$\chi(z) \cdot s = z^n s$ for all $z\in \Cstar$, $s\in \cm_{y_0}$.

The following result is analogous to Corollary 8 in \cite{BeKu:06}, and
can be proved in the same way.

\begin{Lem}\label{l.levimovabletheta}
Let $(w_1, \dots, w_s)$ be a $s$-tuple in $(W_P)^s$ satisfying
equation \eqref{e.codimcondition}. Then

\noindent (1) The section $\theta$ is nonzero on $(P/B_L)^s$ if and only
if $\gL_{w_1} \cdot \cdots \cdot \gL_{w_s} \not= 0 \in H^*(G/P)$;

\noindent (2) The $s$-tuple $(w_1, \dots, w_s)$ is $L_K$-movable
if and only if the restriction of $\theta$ to $(L_K \cap P/B_L)^s$
is not identically zero.
\end{Lem}

We now indicate how to extend the proof of Theorem 15 of \cite{BeKu:06}
to prove the following result, which is the main result of this
section.

\begin{Thm}\label{t.levimovable} Assume that $(w_1, \dots, w_s)
\in (W^P)^s$ satisfies equation \eqref{e.codimcondition}.
Then the following assertions are equivalent:

\noindent (1) The $s$-tuple $(w_1, \dots, w_s) \in (W^P)^s$
is $L_K$-movable;

\noindent (2) $\gL_{w_1} \cdot \cdots \cdot \gL_{w_s} \not= 0 \in H^*(G/P)$
and for all $z\in \fz(\fl_K)$,

\begin{equation} \label{e.centercondition}
((\sum_{j=1}^s \chi_{w_j}) - \chi_1)(z)=0;
\end{equation}

\noindent (3) $\gL_{w_1} \cdot \cdots \cdot \gL_{w_s} \not= 0 \in H^*(G/P)$
and 

\begin{equation} \label{e.rhocondition}
((\sum_{j=1}^s \chi_{w_j}) - \chi_1)(z_\rho)=0;
\end{equation}
\end{Thm}

\begin{proof} To prove (1) implies (2), we assume that 
$(w_1, \dots, w_s)$ is $L_K$-movable.
By part (2) of Lemma \ref{l.levimovabletheta}, the restriction $\hat{\theta}$
of $\theta$ to $(L_K \cap P/B_L)^s$ is nonzero, so in particular
the restriction of the line bundle $\cl$ to $(L_K \cap P/B_L)^s$
has a $L_K \cap P$-invariant section for the diagonal action of
$L_K \cap P$ on $(L_K \cap P/B_L)^s$.
 It follows as in \cite{BeKu:06} that the center $Z(L_K)$ acts
trivially on the fiber $\cl_{eP}$, which implies equation \eqref{e.centercondition}.

It remains to prove that (3) implies (1). For this,
 note that by part
 (1) of  Lemma \ref{l.levimovabletheta},
the first condition in (3) implies that there exists $x=(p_1 B_L, \dots,
p_s B_L) \in (P/B_L)^s$ such that $\theta(x)\not= 0$.
For $u\in U$, the unipotent radical of $P$,
 let $u_0 = \lim_{t\to 0} \gl_\rho (t) u \gl_\rho (t)^{-1}$ and note that
$u_0 \in L_K \cap P$. Let $v=ulB_L \in P/B_L$ with
$u\in U$ and $l\in L$. Since $\gl_\rho \in X_*(Z(L))$, it follows
that $\lim_{t\to 0} \gl_\rho(t) v = u_0 l B_L \in (L_K \cap P)/B_L$.
Hence, for any $y\in (P/B_L)^s$,
 $y_0 := \lim_{t\to 0} \gl_\rho (t) y \in (L_K \cap P/B)^s$.
It follows from Lemma 14, Proposition 10, and equation (15)
 in \cite{BeKu:06} that
$\mu^{\cl}(y, \gl_\rho)= (\chi_1 - (\sum_{j=1}^s \chi_{w_j}))(z_\rho)$
 for each $y \in (P/B_L)^s$. 
Hence, by the assumption \eqref{e.rhocondition} and  Proposition
10(c) in \cite{BeKu:06}, it follows that $\theta(x_0)\not= 0$.  Thus,
by part (2) of Lemma \ref{l.levimovabletheta}, the $s$-tuple
$(w_1, \dots, w_s) \in (W^P)^s$ is $L_K$-movable.
\end{proof}

Let $\tau_K = p_J^2 \in \C^m$, so $(\tau_K)_i = 1$ for $i\in J$, and
$(\tau_K)_i = 0$ for $i\in \{ 1, \dots, m \} - J$.

\begin{Cor} \label{c.lmovproduct}

\noindent (1) A $s$-tuple $(w_1, \dots, w_s)$ of elements in $W^P$ satisfying
\eqref{e.codimcondition} is $L_K$-movable if and only
if 
\begin{equation} \label{e.prodcondlevi}
\gL_{w_1} \odot_{\tau_K} \cdots \odot_{\tau_K} \gL_{w_s} \not= 0 \in H^*(G/P)
\end{equation}
\noindent (2) For $u, v, w\in W^P$ with $l(u) + l(v)= \dim(G/P) + l(w)$, then
$\gL_u \odot_{\tau_K} \gL_v$ has nonzero $\gL_w$ coefficient if and only if the
triple $(u,v, w^*)$ is $L_K$-movable.
\end{Cor}

\begin{proof}
By \eqref{e.codimcondition}, 
$\gL_{w_1} \cdots \gL_{w_s} = d\gL_1$ for some $d\in \Z$.
By Proposition 17(c) in \cite{BeKu:06},
$$
\gL_{w_1} \odot_{\tau_K} \cdots \odot_{\tau_K} \gL_{w_s} = 
\prod_{i=1}^m \tau_i^{((\chi_1 - (\sum_{j=1}^s\chi_{w_j}))(x_i))} d\gL_1.
$$
Since $\fz(\fl_K)=\sum_{i=1. \dots, m; i\not\in J } \C x_i$, 
it is routine to check that this expression is nonzero at $\tau_K$
if and only if condition (2) of Theorem \ref{t.levimovable} is satisfied.
Part (1) of the corollary follows easily from Theorem \ref{t.levimovable}.
 Part (2) follows from (1) using
Lemma 16 (c), (d) of \cite{BeKu:06}.
\end{proof}

\begin{Rem}\label{r.tauKsuffice}
For $\tau \in \C^m$, let $J = \{ i \in \{ 1, \dots, m \} : \tau_i \not= 0 \}$,
 let $K=I\cup J$, and let $L_{\tau}=L_K$.  By equation \eqref{e.fgtvsftpt}, the statement
of the previous corollary is true for any $\tau \in \C^m$
if we replace  $\odot_{\tau_K}$ by $\odot_{\tau}$ and $L_K$ by $L_\tau$.
\end{Rem}

\end{document}